\documentclass[a4paper, 11pt]{article}

\usepackage{graphicx,xfrac}
\usepackage[british]{babel}

\usepackage{calc, xspace}
\usepackage{amsmath, amsthm, amsfonts, amssymb, euscript,mathrsfs}
\usepackage{enumerate}

\usepackage[colorlinks,citecolor=blue,urlcolor=blue]{hyperref}
\usepackage[numbers]{natbib}

\usepackage[textsize=tiny]{todonotes}

\usepackage{a4wide}
\parskip \medskipamount

\usepackage{geometry}
\theoremstyle{plain}
\newtheorem{thm}{Theorem}[section] \newtheorem{prop}[thm]{Proposition}
\newtheorem{lemma}[thm]{Lemma}

\theoremstyle{definition}

\theoremstyle{remark}
\newtheorem{rmk}[thm]{Remark}

\newcommand{\trel}{t_{\mathrm{rel}}}
\newcommand{\tmix}{t_{\mathrm{mix}}}

\def\cS{\mathcal{S}}

\def\tv{_{\mathrm{TV}}}

\def\P{\mathrm{P}}
\def\PP{\mathbb{P}}

\def\EE{\mathbb{E}}
\def\eps{\varepsilon}

\newcommand{\indic}[1]{\mathbf{1}_{\left\{#1\right\}}}

\usepackage{color}

\begin{document}

 \title{A direct comparison  between the mixing time of the interchange process with ``few" particles and independent random walks}
\author{Jonathan Hermon
\thanks{
The University of British Columbia, Department of Mathematics,  1984 Mathematics Road, Vancouver, BC V6T 1Z2, Canada.    
E-mail: {\tt jhermon@math.ubc.ca.} Supported by NSERC grants.}
\and Richard Pymar
\thanks{Department of Economics, Mathematics and Statistics, Birkbeck, University of London, London, WC1E 7HX, UK. E-mail: {\tt r.pymar@bbk.ac.uk}} 
}
\date{}

\maketitle
\begin{abstract}
We consider the interchange process with $k$ particles (denoted $\mathrm{IP}(k)$) on $n$-vertex hypergraphs in which each hyperedge $e$ rings at rate $r_e$.  When $e$ rings, the particles occupying it are permuted according to a random permutation from some \emph{arbitrary} law, where our only assumption is that IP(2) has uniform stationary distribution. We show that $t_{\mathrm{mix}}^{\mathrm{IP}(k)}(\eps)=O_{b}(t_{\mathrm{mix}}^{\mathrm{IP}(2)}(\eps/k))$, where $t_{\mathrm{mix}}^{\mathrm{IP}(i)}(\eps)$ is  the $\eps$ total-variation mixing time of $\mathrm{IP}(i)$,  provided that $kn^{-2}Rt_{\mathrm{mix}}^{\mathrm{IP}(2)}(\varepsilon/k) =O( (\eps/k)^b ) $ for some $b>0$,  where $R=\sum_e r_e|e|(|e|-1)$ is $n(n-1)$ times the particle-particle interaction rate at equilibrium. 

 This has some consequences concerning the validity (in this regime) of conjectures of Oliveira about comparison of the $\eps$ mixing time of $\mathrm{IP}(k)$ to that of $k$ independent particles, each evolving according to $\mathrm{IP}(1)$, denoted RW$(k)$,  and  of Caputo about comparison of the spectral-gap of  $\mathrm{IP}(k)$  to that of a single particle  $\mathrm{IP}(1)=\mathrm{RW}(1)$.

We also show that  $t_{\mathrm{mix}}^{\mathrm{IP}(k)}(\eps) \asymp t_{\mathrm{mix}}^{\mathrm{RW}(1)}(\eps)\asymp t_{\mathrm{mix}}^{\mathrm{RW}(k)}(\eps k/4)$ for all $k \lesssim n^{1-\Omega(1)}$ and all $\eps \le \frac{1}{k} \wedge \frac{1}{4}$    for vertex-transitive graphs of constant degree, as well as for general graphs satisfying a mild (``transience-like") heat-kernel condition. 

In the special case where the particles occupying a hyperedge $e$ are permuted uniformly at random (in $e$) when $e$ rings, we obtain results bounding the spectral gap of $\mathrm{IP}(k)$ in terms of that $\mathrm{RW}(1)$.

 In contrast to recent works on mixing times of IP$(k)$, the proof does not use Morris' chameleon process. It can be seen as a rigorous and direct way of arguing that when the number of particles is fairly small, the system behaves similarly to $k$ independent particles, due to the small amount of interaction between particles.  
 \end{abstract}

 \section{Introduction}
The \emph{interchange process} IP$(k)$ on a finite, connected graph $G=(V,E)$ is the following continuous-time Markov process. In a configuration, each vertex is either occupied by a labelled black  particle, or by an unlabelled white particle such that the number of black particles equals $k\le |V|=:n$. We label the black particles by the set $[k]:=\{1,\ldots,k\}$. For each edge $e$ independently, at the times of a Poisson process of rate $r_e>0$, switch the particles on the endpoints of $e$.  The \emph{exclusion process} EX$(k)$ is similarly defined, except the black particles are also unlabelled.  We will further denote by RW$(k)$ the process of $k$ independent continuous-time random walks on $G$, each with the same transition rates $\{r_e\}_{e\in E}$ (i.e.\ RW$(1)$ equals IP$(1)$ and RW$(k)$ are $k$ independent IP$(1)$). Motivated by conjectures of Oliveira and Caputo (stated below) we are interested in comparing the mixing times and spectral gaps of IP$(k)$ with those of RW$(k)$ and RW$(1)$.

Our interest extends also to IP$(k)$ on hypergraphs. In this process each (hyper)edge $e$ rings independently at rate $r_e$. When an edge rings, some random permutation (not necessarily uniformly distributed) of the vertices in $e$ is applied to the current configuration (that is, the particles currently occupying $e$, including the white ones, are permuted). We stress that we make no assumption on the law of this random permutation of $e$, other than it being the same law at every ring of $e$. We also make the global assumption  that the process IP$(2)$ is irreducible (although when this fails our result holds trivially) and has uniform stationary distribution.
 We suppose throughout that $n\ge 3$. By abuse of terminology, we shall often use the term `hypergraph' to refer also to the associated rates and the rules of the dynamics associated with the edge rings.
 
Of particular interest is the  case in which, when a hyperedge $e$ rings,  the particles currently occupying $e$ are permuted uniformly at random (in $e$). We refer to this setup as a \emph{uniform interchange process}.  Caputo conjectures (see \cite{Cesi}) that the spectral gaps of IP$(k)$ and RW$(1)$ are the same for the uniform interchange processes (for all $k$) -- this is the hypergraph version of the Caputo, Liggett and Richthammer Theorem \cite{Caputo}, a.k.a.\ Aldous' spectral gap conjecture. This provides motivation for comparing the spectral gaps of  IP$(k)$ and  RW$(1)$  for uniform  interchange processes on hypergraphs (and more generally (in the non-uniform case) with that of IP$(2)$). We note that  a uniform interchange process on a hypergraph is reversible w.r.t.\ the uniform distribution for all $k$. More generally, the same holds whenever the law of the permutation associated with each hyperedge $e$ gives each permutation and its inverse the same probability, which is in particular the case when the law is constant on conjugacy classes.

On graphs, the generator of IP$(k)$ is symmetric  and so when IP$(k)$ is irreducible its  time-$t$ law converges as $t\to\infty$ to uniform on the set of possible configurations. We seek to upper-bound the rate of this convergence, measured using the total-variation distance. 
While one should expect the case when $k$ is small to be easier, as there are fewer interactions between different particles in this case, this historically has not been the case.\footnote{Of course an upper bound on the mixing time of $\tmix^{\mathrm{IP}(n)}$ provides an upper bound also on $\tmix^{\mathrm{IP}(k)}$ for $k<n$ by the contraction principle. However, obtaining more refined bounds on $\tmix^{\mathrm{IP}(k)}$ for small $k$ has proven challenging.} The order of the mixing time of IP$(n^d)$ on a $d$-dimensional torus $\mathbb{Z}_n^d$ of side length $n$ was first determined by Yau \cite{Yau} by estimating the log-Sobolev constant  (this gives an upper bound; a lower bound, matching up to a constant factor, was first proven by Wilson \cite{Wilson} for $d=1,2$, and later by Morris in \cite{Morris} for all $d$; see \cite[Theorem 1.4]{HPexclusion} for a general lower bound  which combines ideas from the aforementioned two proofs, together with negative correlation). In \cite{Morris}, Morris introduces the ingenious chameleon process, a process similar to the evolving sets process \cite{evolving}, tailored to handle the complicated dependencies between the particles in the interchange process. This allows determination of the order of the mixing time of IP$(k)$ on  $\mathbb{Z}_n^d$ for all $k$, which offers an improved bound for $k=n^{o(1)}$ compared with \cite{Yau}. For larger $k$ it also offers some improvement, but  only by a constant factor, as well as a better constant dependence on the dimension (from linear to logarithmic). 

While obtaining refined bounds on $\tmix^{\mathrm{IP}(k)}:=\tmix^{\mathrm{IP}(k)}(1/4)$, the $1/4$ mixing time of IP$(k)$, when $k$ is small is one of the main motivations in \cite{Morris}, using the chameleon process on other graphs to obtain refined bounds for small $k$ has proven challenging. Oliveira \cite{Olive} generalises Morris' argument to arbitrary graphs and rates by an elegant use of the negative correlation property enjoyed by the exclusion process on graphs. Alas, his method gives the same upper bound on  $\tmix^{\mathrm{IP}(k)}$ for all $k$. The analysis in \cite{HPexclusion}, which refines that of \cite{Olive} (other than the fact that the setup in  \cite{Olive} is more general), also relies on the chameleon process. More specifically, the chameleon process is analysed using $L_2$ techniques and by exploiting a certain negative association property of the exclusion process on graphs.\footnote{Unfortunately, this property fails to hold for exclusion on hypergraphs. This prevents one from extending the analysis from \cite{HPexclusion} from graphs to hypergraphs and is one of the main obstacles encountered in \cite{CP}.} The most challenging proof in \cite{HPexclusion} is of the refined bound on  $\tmix^{\mathrm{IP}(k)}$ for small $k$; namely, the proof that  $\tmix^{\mathrm{IP}(k)}$ is upper bounded by the upper bound on the $1/k$ $L_2$ mixing time of RW(1) given by the spectral-profile. We do not believe it is possible to prove a stronger result using the chameleon process.\footnote{Recalling that the chameleon process is a variant of the evolving sets process, and that the spectral-profile bound on the mixing time refines the evolving sets isoperimetric-profile bound \cite{spectral}.}

\subsection{Results} 
For a continuous-time Markov process $Q$ we will denote by  $\trel^Q$ and $\tmix^Q(\varepsilon)$ the inverse of the spectral-gap of the process and its $\varepsilon$ total-variation mixing time, respectively. When $\varepsilon=1/4$ we omit it from this notation.

In this paper we present a simple way of analysing $\tmix^{\mathrm{IP}(k)}(\varepsilon)$ for small $k$ which does not rely on the chameleon process. Besides its simplicity, it has the advantage of applying  also for hypergraphs  for arbitrary rates $\{r_e\}_e$ and for arbitrary rules for the law of the random permutation of a hyperedge $e$ when it rings. The argument can be seen as a direct way of making rigorous the intuition that when particles rarely interact with one another the system should evolve similarly to $k$ independent particles.  We emphasize that our approach goes beyond a more na\"ive version of such an argument which requires $k$ to be small enough that with probability bounded away from zero, no particles interact with other particles until they are mixed (after a certain initial burn in period).  
Instead, by considering the behaviour of the $k$th particle conditioned on the rest and exploiting a certain submultiplicativity property (presented in Lemma~\ref{L:submulti}) we are able to extend the result to much larger values of $k$.

Before stating our main (and more general) theorem (see Theorem~\ref{T:main}), we first present some more lucid results. The first concentrates on the case of vertex-transitive graphs, for which we present bounds on the $\eps$ mixing time of IP$(k)$ for $\eps \le 1/k$. In fact, our  argument gives the same bounds (up to a constant factor) on the $1/4$ and the $1/k$ mixing times of IP$(k)$ (this applies to all of our results). The fact that we bound also $\tmix^{\mathrm{IP}(k)}( 1/k)$ rather than just $\tmix^{\mathrm{IP}(k)}=\tmix^{\mathrm{IP}(k)}( 1/4)$ will be important later on in order to derive an upper bound on the relaxation-time and is also used to derive a comparison with $\tmix^{\mathrm{RW}(k)}$. 

We write $o(1)$ for terms which vanish as $n\to\infty$ and $O(1)$ for terms which are bounded from above by a constant. We write $f_n\lesssim g_n$ if $|f_n|/|g_n|=O(1)$ and $f_n\asymp g_n$ if $g_n\lesssim f_n\lesssim g_n$. We write $f_n \asymp_\eta g_n$ if the implicit constant depends on $\eta$. Similarly, we write $C(\eta)$ or $c_{\eta}$ for positive constants depending only on $\eta $. We also write $a \wedge b:=\min \{a,b\}$. 

\begin{thm}[Bound for vertex-transitive graphs]
\label{thm:transitive}
For every $d\in\mathbb{N}$ and $a\in(0,1)$ there exist  $n_0=n_0(a,d)$ and $c=c(a,d)>0$ such that 
for every vertex-transitive graph $G=(V,E)$ equipped with rates $r_e\equiv1$, of size $|V|=n\ge n_0$, of degree $d$,  for all $3 \le k\le n^{a}$  we have that
\begin{equation}
\label{e:VTmaineq1}
 c(a,d)\tmix^{\mathrm{IP}(k)}( \varepsilon)\le  \tmix^{\mathrm{RW}(1)}(\varepsilon) \asymp   \tmix^{\mathrm{RW}(k)}(\varepsilon k/4) \quad \text{for all } \varepsilon\le \frac1{k}\wedge\frac14.
\end{equation}

\end{thm}
It easily follows from \eqref{e:VTmaineq1} that $c'(a,d)\tmix^{\mathrm{IP}(k)}( \varepsilon) \le  \tmix^{\mathrm{RW}(k)}(\varepsilon)$ for all $\eps \le 1/4$. This gives a partial answer to a more general conjecture of Oliveira (see \S\ref{s:Oliveira'sconjecture}).

The proof involves a certain case analysis. Let $D$ be the diameter of the vertex-transitive graph. If $D$ is at least polynomial in $n$ then using an inspired recent approximate group theoretical result of Tessera and Tointon \cite{TandT}, providing finitary quantitative forms of Gromov's and Trofimov's  Theorems, it follows that the graph satisfies a certain technical condition due to Diaconis and Saloff-Coste \cite{moderate}, called ``moderate growth",  and this case is already covered by \cite[Prop.\ 11.1]{HPexclusion}.\footnote{\label{F:AK}The Cayley graph case is due to Breuillard and Tointon \cite{MR3439705}. A finitary version of Gromov's Theorem was first prove by Breuillard, Green and Tao \cite{BGT}. A recent result of Alon and Kozma \cite{Alon} is that for graphs with general rates $\{r_e\}_e$, under mild conditions   $\tmix^{\mathrm{IP}(n)} \lesssim \tmix^{\mathrm{IP}(1)} \log n$.  In the case of vertex transitive graphs of moderate growth, this bound  combined with the analysis from \cite[Thm 1.4 \& \S 11]{HPexclusion} implies that $\tmix^{\mathrm{IP}(n)}\asymp  \tmix^{\mathrm{IP}(1)}(1/n) \asymp   \tmix^{\mathrm{RW}(n)}$ (the implicit constant in the first $\asymp$ depends on $a,b,c,d$, where $D \ge c n^b$)). We also note that while the results in \cite{HPexclusion} are stated for EX($k$), they are all proven for IP($k$) for $k \le n/2$, and are in fact valid for IP$(k)$ when $k \le(1-\delta)n$, but with additional dependence on $\delta$ of some constants.} Hence it suffices to consider the case that $D\le  n^{1/3}$. To treat this case we appeal to a result in an upcoming work of the first author with Berestycki and Teyssier, which  relies on an isoperimetric inequality due Tessera and Tointon \cite{TandTsharp} (which in turn, follows from their estimates on growth of balls from \cite{TandT}), in order to verify the conditions in Theorem~\ref{thm:transiance} below. (We note that similar reasoning has previously been used by Tessera and Tointon in \cite{TandTsharp} to prove some related results).

The next result concerns general graphs satisfying a certain heat-kernel condition. Loosely speaking, this is the condition that the spectral-dimension is at least $2+\eps$. Such a ``transience-like" condition is consistent with the general theme of this paper of bounding the mixing time under regimes in which there are few interactions between particles. We write $p_t(x,y)$ for the time $t$ transition probability from $x$ to $y$ of RW(1).

\begin{thm}[Bound under a `transience-like' heat-kernel condition]
\label{thm:transiance}
For every $d\in\mathbb{N}$, $\theta,a\in(0,1/2)$ and $c>0$  there exists $n_0=n_0(a,d)$ and $C=C(c,d,\theta)$ (both independent of $k,G$ and $n$) such that for every connected graph $G=(V,E)$ equipped with rates  $r_e \equiv 1$, of size $|V|=n\ge n_0$, of maximal degree $d$, satisfying $\trel^\mathrm{RW(1)}\le n^{1-2a}$  and 
\begin{equation}
\label{e:HKsimpler}
\tag{HK-$(\theta)$}
\max_x p_t(x,x) -\frac1{n}\le \frac{c}{t^{1+\theta}}\quad \forall t\ge t_\mathrm{rel}^{\mathrm{RW}(1)} ,
\end{equation} 
 we have
\[
t_\mathrm{mix}^{\mathrm{IP}(k)}(\varepsilon)\le C t_\mathrm{mix}^\mathrm{RW(1)}(\varepsilon)\asymp   \tmix^{\mathrm{RW}(k)}(\varepsilon k/4)
\]
for all $3\le k\le n^a$ and $\varepsilon\le \frac1{k}\wedge\frac14$.

\end{thm}

It easily follows that under the assumptions \eqref{e:HKsimpler} and  $\trel^\mathrm{RW(1)}\le n^{1-2a}$ we have that  $\tmix^{\mathrm{IP}(k)}( \varepsilon) \lesssim_{c,d,\theta}  \tmix^{\mathrm{RW}(k)}(\varepsilon)$ (uniformly) for all $\eps \le 1/4$ and  all $3\le k\le n^a$, provided that $n \ge n_0(a,b,d)$. This is yet another partial progress on Oliveira's conjecture. We remark that we could have instead assumed that   $\trel^\mathrm{RW(1)}\le n^{1-a-b}$ for some $a,b>0$. This allows one to consider larger values of $k$, as $k \le n^{a}$ and the above allows to take a larger value for $a$. However the case that $n^{c} \lesssim  k \lesssim n^{1-c}$ for some $c \in (0,\frac 12)$ is already covered in \cite{HPexclusion}.   We also note that the proof of Theorem~\ref{thm:transiance} uses the negative correlation property of the exclusion process -- a property which does not hold for the process on hypergraphs. As a result, one cannot easily extend this result to hypergraphs.
\begin{rmk}
Our proof shows that for the last inequality to hold for a certain $3\le k\le n^a$ and $0<\eps\le \frac1{k}\wedge\frac14$ we only require \eqref{e:HKsimpler} to hold at a time $t=2\alpha t_\mathrm{mix}^{\mathrm{RW}(1)}(\varepsilon)$ for some constant $\alpha=\alpha(c,d,\theta)$ which is chosen in the proof.

\end{rmk}
We introduce some notation before presenting our main result. For a size $n$ hypergraph with rates $\{r_e\}_e$ we set \[R:=\sum_e r_e|e|(|e|-1).\]  The quantity $R/[n(n-1)]$ is the rate of particle-particle interaction for two particles at equilibrium.\footnote{Note that for $d$-regular hypergraphs with all hyperedges of size $L$ if $r_e\equiv 1/d$, we have that $Rn^{-1}=L-1$. For hypergraphs with maximal hyperedge size $L$ and maximal degree $\Delta$, if  $r_e \equiv 1$ then $Rn^{-1} \le \Delta L$.} The appearance of $\delta$  in the below theorem may at first appear cumbersome, however it arises naturally in the proof as a bound on the expected number of interactions of a certain particle with the rest of the particles during a time interval of length $\tmix^{\mathrm{IP}(2)}(\frac{\eps}{8k})$, after an initial burn-in period of length $\tmix^{\mathrm{IP}(2)}(\frac{\eps}{8k})$. The quantity
$\tmix^{\mathrm{IP}(2)}(\eps)$ for  $\varepsilon\in(0,\frac14\wedge\frac1{k}]$ appearing below has a natural interpretation. It is up to some universal constant comparable to the $\eps k/4$ mixing time of $\lfloor k/2 \rfloor$ independent realizations of IP$(2),$ cf.\ \cite{HPexclusion}.
Recall that we make the global assumption  that the process   IP$(2)$ is irreducible  and has uniform stationary distribution.

 \begin{thm}[Bound for hypergraphs]
 \label{T:main}
 There exists a universal constant $C>0$ such that for every size $n$ hypergraph  and for each $k\ge 3$ and $\varepsilon\in(0,\frac14\wedge\frac1{k}]$ satisfying $\delta= \delta(\eps,k):=8Rkn^{-2}\tmix^{\mathrm{IP}(2)}(\frac{\eps}{8k})<1$ we have that 
 \begin{enumerate}
 \item if $\varepsilon k^{-1}\ge 2\delta$ then
 \[
 \tmix^{\mathrm{IP}(k)}(\varepsilon)\le C\tmix^{\mathrm{IP}(2)}(\varepsilon),
 \]
 \item if $\varepsilon k^{-1}<2\delta$ then
 \[
  \tmix^{\mathrm{IP}(k)}(\varepsilon)\le C\tmix^{\mathrm{IP}(2)}(\varepsilon)\log_{1/\delta}(k/\varepsilon).
 \]
 \end{enumerate}
Equivalently, for all $a>0$, if $\delta(\eps,k) \le (\varepsilon/k)^{a}$ then $\tmix^{\mathrm{IP}(k)}(\varepsilon)\le C \frac{1+a}{a} \tmix^{\mathrm{IP}(2)}(\varepsilon)$. 

Moreover, there exists $C'>0$ such that for all $b \in (0,1] $ there exists $n_0(b)$ such that for every size $n \ge n_0(b)$ hypergraph, for each $k\ge 3$,  if $Rkn^{-2}\tmix^{\mathrm{IP}(2)}(n^{-b})\le n^{-b}$ then 
\begin{equation}
\label{e:Caputomix}
\tmix^{\mathrm{IP}(k)}(\varepsilon)\le C' b^{-1} \tmix^{\mathrm{IP}(2)}(\varepsilon) \quad \text{for all }0<\varepsilon\le \frac14\wedge\frac1{k},   
\end{equation}
and if in addition  \emph{IP}$(2)$ and  \emph{IP}$(k)$ are also reversible then 
\begin{equation}
\label{e:Caputo1}
\trel^{\mathrm{IP}(k)} \le C'b^{-1}\trel^{\mathrm{IP}(2)}.  
\end{equation} 
 \end{thm} 
\begin{rmk}\label{rmk:main}
\begin{enumerate}

\item As we make no assumption on the law of the permutations associated with the hyperedges, other than IP$(2)$ being irreducible with uniform stationary distribution, it need not be the case that IP$(n)$ is irreducible. For instance, consider the case $n=4$ where there is a single hyperedge containing all 4 vertices, and when it rings a random 3-cycle is applied. Since 3-cycles are even permutations  IP$(4)$ is reducible. The irreducibility of IP$(k)$ under the assumption $\delta <1$ is thus a non-trivial consequence of Theorem \ref{T:main}. 
\item The condition $Rkn^{-2}\tmix^{\mathrm{IP}(2)}(n^{-b})\le n^{-b}$ may seem strong, for example if $k=n$ and $r_e\equiv 1$ for all $e$, there are no regular hypergraphs for which the condition holds. However, our main focus is not in this regime; instead we are interested in much smaller values of $k$ for which there do exist hypergraphs satisfying this condition.

\item It follows by the contraction principle (see \cite{aldous}) that the same bounds as in Theorem \ref{T:main} hold for the exclusion process for the same values of $k$.\footnote{In fact, an inspection of the proof reveals that for EX$(k)$ we can replace $\tmix^{\mathrm{IP}(2)}(\cdot)$ by  $\tmix^{\mathrm{EX}(2)}(\cdot) \vee \tmix^{\mathrm{RW}(1)}(\cdot)$ (both in our upper bounds and in the definition of $\delta$ from Theorem \ref{T:main}).}
\end{enumerate}
\end{rmk}

For hypergraphs,  $\tmix^{\mathrm{IP}(2)}(1/4)$ and  $\tmix^{\mathrm{RW}(1)}(1/4)$ can be of different orders, see the example in \cite[Remark 1.5]{CP}. That example also demonstrates that one cannot replace   $\tmix^{\mathrm{IP}(2)}(\eps)$ in Theorem \ref{T:main} by  $\tmix^{\mathrm{RW}(1)}(\eps)$.
 Even for graphs, the proof that $\tmix^{\mathrm{IP}(2)}(1/4)$ and  $\tmix^{\mathrm{RW}(1)}(1/4)$ are comparable is surprisingly difficult \cite{Olive}, and it is not known if   $\tmix^{\mathrm{IP}(2)}(\eps)$ and  $\tmix^{\mathrm{RW}(1)}(\eps)$  are comparable, uniformly for all $\eps  \le 1/4$. One  exception in the graph setup is the case that $\eps \le n^{-\Omega(1)}$ where both quantities are comparable up to a constant factor  to $\trel^{\mathrm{RW}(1)} \log (1/\eps)$ (using \eqref{e:RW1mixtrel} combined with $\trel^{\mathrm{IP(2)}}=\trel^{\mathrm{RW}(1)}$ (the Caputo, Liggett and Richthammer Theorem \cite{Caputo})). Using \eqref{e:trel2trel1} below one can show that     $\tmix^{\mathrm{IP}(2)}(\eps)$ and  $\tmix^{\mathrm{RW}(1)}(\eps)$  are comparable also  for uniform hypergraphs, when  $\eps \le n^{-\Omega(1)}$.

To complement Theorem \ref{T:main}, we are interested in finding general conditions under which the conditions of Theorem \ref{T:main} hold, and under which  $\tmix^{\mathrm{IP}(2)}(\eps)\asymp \tmix^{\mathrm{RW}(1)}(\eps)$ or   $\trel^{\mathrm{IP}(2)}\asymp \trel^{\mathrm{RW}(1)}$. Theorem \ref{thm:transiance} gives one such case, and the following result gives another. 

 \begin{thm}
\label{thm:Caputo}
There exists an absolute constant $C>0$ such that for all $b \in (0,1] $  there exists $n_0(b)$ such that for all uniform interchange processes on a size $n \ge n_0(b) $  hypergraph satisfying   $Rkn^{-2}\tmix^{\mathrm{RW}(1)}(n^{-b})\le    n^{-b}$ we have that
\begin{equation}
\label{e:Caputo12uniform}
\trel^{\mathrm{IP}(k)} \le C b^{-1} \trel^{\mathrm{RW}(1)}.
\end{equation}

Moreover, regardless of the value of $R$, for  a uniform interchange process on a finite hypergraph we always have that
\begin{equation}
\label{e:trel2trel1}
\trel^{\mathrm{IP}(2)} \le C \trel^{\mathrm{RW}(1)}.
\end{equation}  
\end{thm}
Equation \eqref{e:Caputo12uniform}  can be seen as a partial progress on a conjecture of Caputo \cite{Cesi} that for uniform interchange process on hypergraphs,  we have that $\trel^{\mathrm{IP}(n)}=\trel^{\mathrm{RW}(1)}$. We note that in \eqref{e:Caputo1} we consider a more general class of interchange processes, by not requiring the permutations chosen to be uniformly distributed. However, \eqref{e:Caputo1} and \eqref{e:Caputo12uniform}  are of course weaker than Caputo's conjecture, as they do not apply for all $k$ and include some absolute constant.

The proof of \eqref{e:trel2trel1} uses a comparison of Dirichlet forms. One obstacle is that $\mathrm{IP}(2)$ and $\mathrm{RW}(2)$ (as well as $\mathrm{RW}(1)$) do not have the same state space. Moreover, some care is required to avoid dependence on the maximal degree, the maximal size of a hyperedge  and on $\max_{e,e' \in E}\frac{r_e}{r_e'} $ in the constant in the right-hand side of \eqref{e:trel2trel1}.

\subsection{Oliveira's conjecture - comparing with independent particles}
\label{s:Oliveira'sconjecture}
In \cite{Olive} Oliveira showed the existence of a universal constant $C$ such that for general graphs (but not hypergraphs) and rates $\max_{k \le n/2} \tmix^{\mathrm{EX}(k)}(\varepsilon)\le C\tmix^{\mathrm{RW}(1)}\log(n/\varepsilon)$  for all $\varepsilon\in(0,1)$, and conjectured that for all $\varepsilon\in(0,1)$ and $k \le n$, 
\begin{align}\label{e:Oconj}
 \quad\tmix^{\mathrm{EX}(k)}(\varepsilon)\le C\tmix^{\mathrm{RW}(k)}(\varepsilon)
\end{align}
(see \cite[Conjecture 2]{IPhypercube} and \cite[Question 1.5]{HPexclusion} for related problems; see Footnote \ref{F:AK} for a recent related result). As mentioned above, our results verify this in various setups for `small' $k$.

In \cite{HPexclusion} we proved an upper-bound on $\tmix^{\mathrm{EX}(k)}$ that was within a multiplicative factor of $\log\log n$ of this conjecture for all $k=n^{\Omega(1)}$ such that $k \le n/2$ (for regular graphs with $r_e\equiv 1$). Further, we demonstrated in certain situations that the conjecture holds; for example if $n^{c} \lesssim k\lesssim n^{1-c}$ for $c\in(0,1/2)$ (in this case the constant $C$ in \eqref{e:Oconj} depends on $c$). Other examples when $k$ is not assumed to be small are when the spectral-gap is at most $(\log n)^{-4}$ or when the degree is at least logarithmic. These bounds (and in fact all results on EX$(k)$ in \cite{HPexclusion}) are valid also for $\tmix^{\mathrm{IP}(k)}$ when $k \le n(1-c)$ (for any constant $c \in (0,1/2]$, possibly with additional dependence of the constant $C$ on $c$). It is interesting to note that the graphs for which \cite{HPexclusion} does not offer sharp bounds for large $k$ are, as said above, ones with fairly large spectral-gaps (at least  $(\log n)^{-4}$). For such graphs the condition on $k$ in Theorems \ref{thm:transiance} and \ref{T:main} is milder.

 \subsection*{Acknowledgements} We are grateful to Pietro Caputo, Roberto Oliveira and Justin Salez for stimulating discussions.

\section{Preliminaries}
\subsection{Notation and basic definitions}
For a finite set $\Omega$  we denote by $\mathcal{S}_\Omega$ the group of permutations of elements in $\Omega$. For  $k\le |\Omega|$, we write $(\Omega)_k$ for the set of $k$-tuples of distinct elements from $\Omega$. For ${\bf x} \in (\Omega)_k$ we denote $\mathbf{ O}({\bf x})=\{{\bf x}_i:i \in [k] \}$. For a random variable $X$ we write $\mathcal{L}[X]$ for the law or distribution of $X$. The total-variation distance between two distributions $\mu$ and $\nu$ is defined as 
\[
\|\mu-\nu\|\tv := \sum_{a}(\mu(a)-\nu(a))_+=\frac12\sum_a|\mu(a)-\nu(a)|=\inf_{(X,Y):\mathcal{L}[X]= \mu, \mathcal{L}[Y]=\nu}\PP[X\neq Y],
\]
where, in the last equality, the infimum is over all couplings $(X,Y)$ of $(\mu,\nu)$.

For  $\omega\in \Omega$ and a continuous-time irreducible Markov process $(X_t^\omega)_{t\ge0}$ with state space $\Omega$ and satisfying $X_0^\omega=\omega$, we define its $\varepsilon$ total-variation mixing time as
\[
\tmix^X(\varepsilon):=\inf\big\{t\ge0:\,\max_{\omega\in \Omega}\|\mathcal{L}[X_t^\omega]-\pi\|\tv\le \varepsilon\big\},
\]
where $\pi$ denotes the stationary distribution of the process.

\subsection{Graphical construction}
We present a construction of a random walk and the interchange process. One important feature of this construction is it places these processes on the same probability space which allows us to directly relate them. Graphical construction for the interchange process on graphs is classical, see \cite{Liggettbook1}. Our construction for this process on hypergraphs is the same as that appearing in Connor-Pymar.

We take the state space to be $\cS_V$. In this notation, the particles are labeled by the set $V$. We think of  $\sigma(v)$ for $\sigma \in \cS_V$ and $v \in V$ as the location of the particle labelled $v$ in configuration~$\sigma$.

The first step is to construct a sequence of independent edge choices, that is, a sequence $(e_n)_{n\in\mathbb{N}}$ with the property that for each $e\in E$, $\mathbb{P}[e_n=e]\propto r_e$. Next, we require a sequence of permutations choices. Given $(e_n)_{n\in\mathbb{N}}$, we construct a sequence $(\sigma_n)_{n\in\mathbb{N}}$ such that for each $n\in\mathbb{N}$, $\sigma_n\in\mathcal{S}_{e_n}$. Finally we determine the jump times. Let $\Lambda$ be a Poisson process of rate $\sum_e r_e$. For $0<s<t$ denote by $\Lambda[s,t]$ the number of points of $\Lambda$ in interval $[s,t]$ and define a permutation $I_{[s,t]}:V\to V$ associated with time interval $[s,t]$ to be the composition of permutations occurring during this time:
\[
I_{[s,t]}=\sigma_{{\Lambda[0,t]}}\circ\sigma_{{\Lambda[0,t]-1}}\circ\cdots\circ\sigma_{{\Lambda[0,s)+1}}.
\]
We set $I_t:=I_{[0,t]}$ for each $t>0$. These functions can be lifted to functions on $(V)_k$ by setting $I_{[s,t]}({\bf x})=(I_{[s,t]}({\bf x}_1),I_{[s,t]}({\bf x}_2),\ldots,I_{[s,t]}({\bf x}_n))$, for ${\bf x}\in(V)_k$. In this construction, $I_{[s,t]}(a)$ is the location at time $t$ of the particle that occupied $a$ at time $s$. We have the following consequence:

\begin{prop}[Proof omitted]
Fix $s\ge0$. Then 
\begin{enumerate}
\item for each $x\in V$, the process $(I_{[s,t]}(x))_{t\ge0}$ is a random walk started from $x$,
\item for each ${\bf x}\in (V)_k$, the process $(I_{[s,t]}({\bf x}))_{t\ge0}$ is a $k$-particle interchange process started from ${\bf x}$.
\end{enumerate}
\end{prop}
\subsection{Some auxiliary results}
\begin{lemma}
\label{lem:RWkvsRW1auxcalculaiton}
\begin{equation}
\label{e:kvs1}
\forall k \ge 3 ,\, \eps \in (0,1/4), \quad  \frac 12 \tmix^{\mathrm{RW}(1)}(4\eps / k)  \le \tmix^{\mathrm{RW}(k)}(\eps) \le \tmix^{\mathrm{RW}(1)}(\eps / k). 
\end{equation}
Moreover,  for an interchange process on a size $n$  hypergraph such that \emph{RW(1)} is reversible and has  uniform  stationary distribution, if for some $b \in (0,1]$ such that $n^{-b} \le 1/4$ and $C \ge 1$ we have that $\tmix^{\mathrm{IP}(k)}(n^{-b}  ) \le C \tmix^{\mathrm{RW}(1)}(n^{-b}  )$ then for all $\eps \in (0,n^{-b})$
\begin{equation}
\label{e:smalleps1}
\tmix^{\mathrm{IP}(k)}(\eps ) \le 32Cb^{-1} \tmix^{\mathrm{RW}(1)}(\eps ),
\end{equation}
and provided that \emph{IP}$(k)$ is also reversible then we  have that
\begin{equation}
\label{e:smalleps4}
\trel^{\mathrm{IP}(k)}\le 32Cb^{-1} \trel^{\mathrm{RW}(1)}.
\end{equation}
Similarly, for an interchange process on a size $n$ hypergraph  such that \emph{IP(2)} is reversible, irreducible and  has  uniform stationary distribution,  if for some $b \in (0,1]$ such that $n^{-b} \le 1/4$ and $C \ge 1$ we have that  $\tmix^{\mathrm{IP}(k)}(n^{-b} ) \le C \tmix^{\mathrm{IP}(2)}(n^{-b})$ then for all $\eps \in (0,n^{-b})$
\begin{equation}
\label{e:smalleps2}
\tmix^{\mathrm{IP}(k)}(\eps ) \le 64Cb^{-1} \tmix^{\mathrm{IP}(2)}(\eps ),
\end{equation}
and provided that \emph{IP}$(k)$ is also reversible then we  have that
\begin{equation}
\label{e:smalleps3}
\trel^{\mathrm{IP}(k)}\le 64Cb^{-1} \trel^{\mathrm{IP}(2)}.
\end{equation}
\end{lemma}
\begin{proof} The first display is Equation (18) from \cite{HPexclusion}. For \eqref{e:smalleps1} we use the general relations between $\tmix$ and $\trel$ (here we rely on reversibility and on the uniform distribution being stationary) 
\begin{equation}
\label{e:RW1mixtrel}
\trel^{\mathrm{RW}(1)} \log \left(\frac{1}{2\eps} \right) \le \tmix^{\mathrm{RW}(1)}(\eps ) \le \trel^{\mathrm{RW}(1)} \log (n/\eps) 
\end{equation} \cite[Lemma 20.11, Theorem 20.6 and (4.43)]{levin} as well as submultiplicativity of mixing times \cite[(4.29)]{levin} (i.e.\ $\tmix(\delta^{\ell}) \le \ell \tmix(\delta/2)$ for all $\delta>0$ and $\ell \in \mathbb{N}$) to deduce that for all $m \in \mathbb{N}$,  \[ \tmix^{\mathrm{RW}(1)}\left(\frac{1}{2n^{2mC(1+b)}}\right) \le\tmix^{\mathrm{RW}(1)}\left(\frac{1}{n^{3mC(1+b)}}\right)   \le  8Cb^{-1}  \tmix^{\mathrm{RW}(1)}\left(\frac{1}{2n^{mb}}  \right) \le 16Cb^{-1}\tmix^{\mathrm{RW}(1)}(n^{-mb}  ),\] and that  
 \begin{equation}
\label{e:RWkRW1trelsmalleps}
\begin{split} 
  \tmix^{\mathrm{IP}(k)}(n^{-mb}  ) & \le m\tmix^{\mathrm{IP}(k)}(n^{-b}/2  ) \le 2m\tmix^{\mathrm{IP}(k)}(n^{-b}  ) \le 2mC\tmix^{\mathrm{RW}(1)}(n^{-b}  ) \\ & \le 2mC \trel^{\mathrm{RW}(1)} \log (n^{1+b})= \trel^{\mathrm{RW}(1)} \log (n^{2mC(1+b)}) \\ & \le  \tmix^{\mathrm{RW}(1)}(n^{-2mC(1+b)}/2 ) \le 16Cb^{-1}\tmix^{\mathrm{RW}(1)}(n^{-mb}  ).
 \end{split}
 \end{equation}
Finally, if $\eps \in (n^{-mb}  ,n^{-(m+1)b}  )$ then by monotonicity and submultiplicativity
\[\tmix^{\mathrm{RW}(k)}(\eps  ) \le \tmix^{\mathrm{RW}(k)}(n^{-(m+1)b}    )  \le 2  \tmix^{\mathrm{RW}(k)}(n^{-mb}    ) \le 32Cb^{-1}\tmix^{\mathrm{RW}(1)}(n^{-mb}  )  \le 32Cb^{-1}\tmix^{\mathrm{RW}(1)}(\eps).      \]
This concludes the proof of \eqref{e:smalleps1}. We now prove \eqref{e:smalleps4}.
By the general relations between $\tmix$ and $\trel$, and equation~\eqref{e:smalleps1},\[\trel^{\mathrm{IP}(k)}\le \frac{ \tmix^{\mathrm{IP}(k)}(\eps  )}{|\log(2 \eps)|} \le\frac{ 32Cb^{-1} \tmix^{\mathrm{RW}(1)}(\eps )}{|\log(2 \eps)|} \le\frac{ 32Cb^{-1} \trel^{\mathrm{RW}(1)} \log(n/\eps)}{|\log(2 \eps)|}.   \] 
Taking the limit as $\eps \to 0$ concludes the proof of
\eqref{e:smalleps4}.
 
 For \eqref{e:smalleps2} we use the general relations between $\tmix$ and $\trel$ (here again we rely on reversibility and on the uniform distribution being stationary),
\begin{equation}
\label{e:IP2mixtrel}
\trel^{\mathrm{IP}(2)} \log \left(\frac{1}{2\eps} \right) \le \tmix^{\mathrm{IP}(2)}(\eps ) \le \trel^{\mathrm{IP}(2)} \log (n^{2}/\eps) 
\end{equation}
as well as submultiplicativity of mixing times to deduce that for all $m \in \mathbb{N}$, 
 \begin{equation}
\label{e:RWkRW1trelsmalleps2}
\begin{split} 
  \tmix^{\mathrm{RW}(k)}(n^{-mb}  ) & \le 2mC\tmix^{\mathrm{RW}(1)}(n^{-b}  ) \le 2mC \trel^{\mathrm{IP}(2)} \log (n^{2+b}) \le 32Cb^{-1}\tmix^{\mathrm{IP}(1)}(n^{-mb}  ). \\ & 
 \end{split}
 \end{equation}
The proof of  \eqref{e:smalleps2} is concluded as that of \eqref{e:smalleps1}. Finally, the proof of \eqref{e:smalleps3} is analogous to that of \eqref{e:smalleps4} and is hence omitted.  
\end{proof}

\section{Proof of bound for hypergraphs: Theorem~\ref{T:main}}

We shall say that two particles \emph{interact} at time $t$ if they occupy some vertices $u$ and $v$ at time $t_-$ (i.e.\ at some time interval $[t-\eps,t)$) and at time $t$ an edge containing $u$ and $v$ rings.   For $s>0$ and ${\bf z}\in(V)_{k-1}$, we denote by $\Gamma_s$ the set of c\`adl\`ag sample paths of a $(k-1)$-particle interchange process up to time $s$ and  by $\Gamma_s^{{\bf a}}\subset\Gamma_s$ those paths which start at configuration ${\bf a}\in(V)_{k-1}$. We also let $J_s=J_s({\bf x})$ be the event that the $k$th particle avoids interacting with the other $k-1$ particles during time interval $[s,2s]$ when we initialise from configuration ${\bf x}\in(V)_k$.

For $s\ge0$, ${\bf z} \in(V)_{k-1}, \gamma\in\Gamma_{2s}^{{\bf z}}$,  $x\in V\setminus {\bf O}({\bf z})$ and $c\in V$, we define a law \[\mu_s(\bullet)=\mu_s^{\gamma,x,c}(\bullet)=
\PP[I_{2s}(x)\in \bullet\mid I_s(x)=c,\,(I_t({\bf z}))_{0\le t\le 2s}=\gamma],\]
that is, $\mu_s$ is the law of the $k$th particle at time $2s$ of a $k$-particle interchange process conditioned on the trajectory of the first $k-1$ particles, and on the location of the $k$th particle at time $s$. (Note that, by the Markov property, $\mu_s^{\gamma,x,c}$ does in fact not depend on $x$.)

 \begin{lemma}\label{L:mutvbound}
For all $s\ge0$, ${\bf z}\in(V)_{k-1}, \gamma\in\Gamma_{2s}^{{\bf z}}$,  $x\in V\setminus{\bf O}({\bf z})$ and $c\in V$,
\[
\|\mu_s^{\gamma,x,c}-\mathcal{L}[I_s(c)]\|\tv\le 1-\PP[J_s(({\bf z},x))\mid I_s(x)=c,\, (I_t({\bf z}))_{0\le t\le 2s}=\gamma].
\]
 \end{lemma}
\begin{proof}
From the definition of total-variation and using that for $a,b,c\in\mathbb{R}_+$, $(a+b-c)_+\le a+(b-c)_+$,
\begin{align}
&\|\mu_s^{\gamma,x,c}-\mathcal{L}[I_s(c)]\|\tv\notag\\&=\sum_{a\in V}\left( \mu_s^{\gamma,x,c}(a)-\PP[I_s(c)=a] \right)_+\notag\\\begin{split}
&\le \sum_{a\in V}\big( \PP[I_{2s}(x)=a,\,J_s(({\bf z},x))\mid I_s(x)=c,(I_t({\bf z}))_{0\le t\le 2s}=\gamma]-\PP[I_s(c)= a] \big)_+\\
&\phantom{\le}+1-\PP[J_s(({\bf z},x))\mid I_s(x)=c,(I_t({\bf z}))_{0\le t\le 2s}=\gamma],\end{split}\label{eq:mutv}
\end{align}
where the last term comes from $\sum_{a \in V}\PP[I_{2s}(x)=a,\,J_s(({\bf z},x))^{c}\mid I_s(x)=c,(I_t({\bf z}))_{0\le t\le 2s}=\gamma]$.

Next we argue that $\PP[I_{2s}(x)=a,\,J_s(({\bf z},x))\mid I_s(x)=c,(I_t({\bf z}))_{0\le t\le 2s}=\gamma] \le \PP[I_s(c)= a]$ for all $a$. The intuition is that having to avoid interacting with the trajectory $\gamma$ during $[s,2s]$ imposed by $J_s(({\bf z},x))$ and the conditioning $(I_t({\bf z}))_{0\le t\le 2s}=\gamma$ can only decrease the chance of reaching any given target vertex $a$ at time $2s$. To prove this we need some additional notation.
For $a,b\in V$ and $s>0$, let $\Gamma_s^{\mathrm{RW}}(a,b)$ be the set of c\`adl\`ag sample paths of a random walk up to time $s$ which starts at $a$ and terminates at $b$. Further, for any $\gamma\in\Gamma_s$, let $\Gamma_{s,\gamma}^\mathrm{RW}(a,b)\subseteq \Gamma_s^\mathrm{RW}(a,b)$ be those sample paths which avoid interacting with $\gamma$ (that is, trajectories of the random walk which do not interact with any of the $k-1$ particles moving according to $\gamma$). Then for $\gamma\in\Gamma_s$ and $a,c\in V$, we have
\begin{align*}
&\PP[I_{2s}(x)=a,\,J_{s}(({\bf z},x))\mid I_s(x)=c,\,(I_t({\bf z}))_{s\le t\le 2s}=\gamma]\\&=\PP[(I_{[s,t]}(c))_{s\le t\le 2s}\in \Gamma_{s,\gamma}^\mathrm{RW}(c,a)]\le \PP[(I_{[s,t]}(c))_{s\le t\le 2s}\in \Gamma_{s}^\mathrm{RW}(c,a)]\\
&=\PP[I_s(c)=a].
\end{align*}Plugging this into \eqref{eq:mutv} gives the claimed inequality.
\end{proof}

For Lemma~\ref{L:mutvbound} to be useful we need to lower-bound the probability of $J_s$:
\begin{lemma}\label{L:probJ}
Fix $\varepsilon\in(0,1)$ and let $s=\tmix^{\mathrm{IP}(2)}(\tfrac{\varepsilon}{16k^2})$. Then for all 
$k\ge 2$,
\[\min_{{\bf x}\in (V)_k}\P[J_s({\bf x})]\ge1-\frac{\varepsilon}{16k}-\frac{sk}{n^2}\sum_er_e|e|(|e|-1).\]
\end{lemma}
\emph{Proof.}
By a union bound \begin{align}\label{eq:Jdecomp}
\PP[J_s^c]\le\sum_{i=1}^{k-1}\PP[J_{s,i}^c],
\end{align} where $J_{s,i}$ is the event that the $k$th particle avoids interacting with the $i$th particle during time interval $[s,2s]$. We will use a coupling argument to upper-bound $\PP[J_{s,i}^c]$ for each $i\in\{1,\ldots,k-1\}$. Specifically, we couple the pair ($i$th and $k$th particles) with a pair started from time 0 according to the stationary distribution of process IP(2). The chosen coupling is one which satisfies the coupling equality in the definition of total-variation. Observe that, crucially, by using a union bound, we can use $k-1$ different couplings (which need not be related to one another in any way), each of which involves just 2 particles.  
  Let $A_{s,i}$ denote the event that the coupling of the $i$th and $k$th particles is successful at time $s$. Then we can write
 \begin{align}\label{eq:JA}
 \PP[J_{s,i}^c]\le \PP[J_{s,i}^c,\,A_{s,i}]+\PP[A_{s,i}^c].
 \end{align} 
 Let $(y,x)$ be the initial location of the $(i,k)$th particles. Since $s= \tmix^{\mathrm{IP}(2)}(\frac{\varepsilon}{16k^2})$, for $k\ge 3$ we have \begin{align*}
 \PP[A_{s,i}^c]&=\|\mathcal{L}[I_s((y,x))]-\pi^{\mathrm{IP}(2)}\|\tv\le \frac{\varepsilon}{16 k^2}.
\end{align*}

We also need to upper-bound $\PP[J^c_{s,i},\,A_{s,i}]$.  For $x,y\in V$, let $T_{s}(x,y)$ denote the number of times that two particles evolving as  IP(2) started from vertices $x$ and $y$ interact during time interval $[0,s]$. By Markov's inequality we have the bound
\[
\PP[J^c_{s,i},\,A_{s,i}]\le \sum_{(x,y)\in(V)_2}\frac1{n(n-1)}\EE[T_{s}(x,y)].
\]
We can bound this expectation via:
\begin{align*}
\sum_{(x,y)\in(V)_2}\frac1{n(n-1)}\EE[T_{s}(x,y)]&=\sum_{(x,y)\in (V)_2}\frac1{n(n-1)}\EE\Big[\int_0^s \sum_{\substack{e:\\I_t(x),I_t(y)\in e}}r_e\,\mathrm{d}t\Big]\\
&=\sum_e r_e \frac1{n(n-1)}\int_0^s \EE\Big[\sum_{x\in V}\indic{I_t(x)\in e}\sum_{\substack{y\in V\\y\neq x}}\indic{I_t(y)\in e}\,\mathrm{d}t\Big]\\
&\le \sum_e r_e \frac1{n(n-1)}\int_0^s |e|(|e|-1) \,\mathrm{d}t\\
&= \frac{s}{n(n-1)}\sum_e r_e|e|(|e|-1).
\end{align*}
Putting the two bounds into \eqref{eq:JA} and using \eqref{eq:Jdecomp}, we obtain
\[
\PP[J_s^c]\le (k-1)\left(\frac{\varepsilon}{16 k^2}+\tfrac{s}{n(n-1)}\sum_e r_e|e|(|e|-1)\right)\le \frac{\varepsilon}{16k }+\frac{sk}{n^2}\sum_er_e|e|(|e|-1). \; \qed
\]

Before stating the next lemma, we define  
\[\bar d_k(t):=\max_{{\bf w}\in (V)_{k-1},u,v\in V\setminus{{\bf O}({\bf w})}}\|\mathcal{L}[I_{t}(({\bf w},u))]-\mathcal{L}[I_{t}(({\bf w},v))]\|\tv.\]In the case $k=1$ this reduces to $\max_{u,v\in V}\|\mathcal{L}[I_t(u)]-\mathcal{L}[I_t(v)]\|\tv$.

 The next lemma formalises the following idea: if the $k$th particle in an interchange process is unlikely to interact  with any of the other $k-1$ particles for time $s$ sufficiently large then, conditionally on the trajectory of the first $k-1$ particles, the $k$th particle will be close to mixed. The idea of the lemma is that the usual submultiplicativity property of the worst case distance from equilibrium  can be extended to the notion $\bar d_k(t)$, provided one only considers couplings of $\mathcal{L}[I_{t}(({\bf w},u))]$ and $\mathcal{L}[I_{t}(({\bf w},v))]$ which take the same value in the first $k-1$ coordinates.
\begin{lemma} \label{L:submulti}
For any $s,t\ge0$, \[\bar d_k(s+t)\le \bar d_k(t)\left(2\max_{{\bf x}\in (V)_k}\left(1-\PP[J_{s/2}({\bf x})]\right)+\bar d_1(s/2)\right).\]
\end{lemma}
\begin{proof}
Let ${\bf z}\in (V)_{k-1}, x,y\in V\setminus{\bf O}({\bf z})$ and write ${\bf x}=({\bf z},x),\,{\bf y}=({\bf z},y)$. Then for any $s,t\ge0$, 
\[
\|\mathcal{L}[I_{s+t}({\bf x})]-\mathcal{L}[I_{s+t}({\bf y})]\|\tv\le \mathbb{E}\left[\|\mathcal{L}[I_{t}({\bf x}(s))]-\mathcal{L}[I_{t}({\bf y}(s))]\|\tv\right]
\]
for any coupling $({\bf x}(s),{\bf y}(s))=(({\bf z}(s),x(s)),({\bf z}(s),y(s)))$ where $\mathcal{L}[{\bf x}(s)]=\mathcal{L}[I_s({\bf x})]$ and $\mathcal{L}[{\bf y}(s)]=\mathcal{L}[I_s({\bf y})]$. The coupling we choose will be one which keeps the first $k-1$ coordinates matched in the two processes (i.e evolves the first $k-1$ particles identically) and moreover, this coupling will depend on the trajectory of the first $k-1$ particles.\footnote{We clarify that we do not couple the dynamics performed by the $k$th particles in the two systems by time $s$. We only couple them at time $s$, in a manner that depends on the trajectories of the rest of the $k-1$ particles by time $s$.} Note that the quantity inside the expectation is zero if ${\bf x}(s)={\bf y}(s)$ and is always bounded by $\bar d_k(t)$ (as we keep the first $k-1$ coordinates equal). 
Hence 
\[
\bar d_k(s+t)\le \bar d_k(t)\PP[{\bf x}(s)\neq{\bf y}(s)].
\]

By our choice of coupling we can write
\begin{align*}
\PP[{\bf x}(s)\neq{\bf y}(s)]&=\EE\left[\PP[{\bf x}(s)\neq {\bf y}(s)\mid x(s/2),y(s/2),({\bf z}(t))_{0\le t\le s}]\right]\\
&=\EE\left[\PP[x(s)\neq y(s)\mid x(s/2),y(s/2),({\bf z}(t))_{0\le t\le s}]\right].
\end{align*}
Given the trajectory $({\bf z}(t))_{0\le t\le s}$, the coupling we choose is that which attains equality in the definition of total-variation, that is, the one which allows us to write the above as
\[
\PP[x(s)\neq y(s)\mid x(s/2),y(s/2),({\bf z}(t))_{0\le t\le s}]=\EE\left[\|\mu-\nu\|\tv\right],
\]
where $\mu=\mu_s^{({\bf z}(t))_{0\le t\le s},x,x(s/2)}$ and $\nu=\mu_s^{({\bf z}(t))_{0\le t\le s},y,y(s/2)}$. Next, by the triangle inequality we have
\begin{align}\label{eq:tvtriangle}\begin{split}
\EE\left[\|\mu-\nu\|\tv\right]&\le \EE\left[\|\mu-\mathcal{L}[I_{s/2}(x({s/2}))]\|\tv\right]+\EE\left[\|\nu-\mathcal{L}[I_{s/2}(y({s/2}))]\|\tv\right]\\
&\phantom{\le}+\max_{x,y}\|\mathcal{L}[I_{s/2}(x)]-\mathcal{L}[I_{s/2}(y)]\|\tv.\end{split}
\end{align}
The first two expectations on the right-hand side can be bounded using Lemma~\ref{L:mutvbound}:
\begin{align*}
\EE\left[\|\mu-\mathcal{L}[I_{s/2}(x({s/2}))]\|\tv\right]&\le 1-\EE\left[\PP[J_{s/2}(({\bf z},x))\mid x({s/2}),\, ({\bf z}(t))_{0\le t\le s}]\right]\\&=1-\PP[J_{s/2}(({\bf z},x))],
\end{align*}
and similarly
\[
\EE\left[\|\nu-\mathcal{L}[I_{s/2}(y({s/2}))]\|\tv\right]\le 1-\PP[J_{s/2}(({\bf z},y))].
\]
The third term on the right-hand side of \eqref{eq:tvtriangle} is simply $\bar d_1(s/2)$.
\end{proof}

\begin{proof}
[Proof of Theorem~\ref{T:main}]
For the first part of the statement, we need to show that  if $s\ge c\tmix^{\mathrm{IP}(2)}(\varepsilon) $ in the case $\varepsilon k^{-1}\ge 2\delta$ and if $s\ge c\tmix^{\mathrm{IP}(2)}(1/k)\log_{1/\delta}(k/\varepsilon)$ in the case $\varepsilon k^{-1}<2\delta$, for some universal $c>0$, then
\begin{align}\label{e:tvbound}
\max_{{\bf x},{\bf y}\in (V)_k}\|\mathcal{L}[I_{s}({\bf x})]-\mathcal{L}[I_{s}({\bf y})]\|\tv\le \varepsilon.
\end{align}
In order to apply Lemma~\ref{L:submulti} we must reduce the above total-variation distance to one between initial configurations which differ in a single coordinate. This is achieved via the triangle inequality. 

Suppose ${\bf x},\,{\bf y}\in(V)_k$ are arbitrary. Note that there exists a sequence ${\bf x}=:{\bf x}_0,{\bf x}_1,\ldots,{\bf x}_r:={\bf y}$ for some $r\le k$ with ${\bf x}_i\in (V)_k$ and such that ${\bf x}_i$ differs in one coordinate from ${\bf x}_{i+1}$ for all $0\le i<r$. Hence for any $s>0$,
\begin{align}
\|\mathcal{L}[I_s({\bf x})]-\mathcal{L}[I_s({\bf y})]\|\tv\le
 \sum_{i=1}^r\|\mathcal{L}[I_s({\bf x}_{i-1})]-\mathcal{L}[I_s({\bf y}_{i-1})]\|\tv.\label{e:reduce}
\end{align}

So now suppose that ${\bf x}=({\bf z},x),\,{\bf y}=({\bf z},y)\in(V)_k$ differ in just one coordinate, which, without loss of generality, we assume is the $k$th coordinate.
Then by repeated application of Lemma~\ref{L:submulti}, for any $m\in\mathbb{N}$ and $t>0$,
\begin{align*}
\|\mathcal{L}[I_{mt}({\bf x})]-\mathcal{L}[I_{mt}({\bf y})]\|\tv\le \left(2\max_{{\bf x}\in (V)_k}\PP[J_{t/2}^c({\bf x})]+\bar d_1(t/2)\right)^m.
\end{align*}

We now set $t=2\tmix^{\mathrm{IP}(2)}(\frac{\varepsilon}{16k^2})$ (which by submultiplicativity of mixing times is at most $ 20\tmix^{\mathrm{IP}(2)}(\varepsilon)$, since $\varepsilon\le \frac1{k}\wedge\frac14$)  so that $\bar d_1(t/2)\le \varepsilon/(8k^2)$. By Lemma~\ref{L:probJ} we thus have
\[
\|\mathcal{L}[I_{mt}({\bf x})]-\mathcal{L}[I_{mt}({\bf y})]\|\tv\le \left(\frac{\varepsilon}{4k}+\frac{tk}{n^2}\sum_er_e|e|(|e|-1) \right)^m\le \left(\frac{\varepsilon}{4k}+\frac{\delta}{2}\right)^m,
\]where the last inequality follows by noting
\begin{align*}\frac{\delta}{2}&=4kn^{-2}\tmix^{\mathrm{IP}(2)}(\varepsilon/(8k))\sum_e r_e|e|(|e|-1) \ge 2kn^{-2}\tmix^{\mathrm{IP}(2)}(\varepsilon/(16k^2))\sum_e r_e|e|(|e|-1) \\ &=tkn^{-2}\sum_e r_e|e|(|e|-1).
\end{align*}
If $\varepsilon k^{-1}\ge 2\delta$  we take $m=1$. Otherwise we take $m=\lceil \log_{1/\delta}(k/\varepsilon)\rceil$ (recall our assumption that $\delta\le 1$). In each case taking $s=mt$ in \eqref{e:reduce} we deduce that for arbitrary ${\bf x},\,{\bf y}\in (V)_k$,
\[
\|\mathcal{L}[I_s({\bf x})]-\mathcal{L}[I_s({\bf y})]\|\tv\le \varepsilon,
\]
which completes the proof of \eqref{e:tvbound}.

It remains to prove~\eqref{e:Caputomix} and~\eqref{e:Caputo1}. Using submultiplicativity of mixing times, provided that $n \ge n_0(b)$ the condition  $Rkn^{-2}\tmix^{\mathrm{IP}(2)}(n^{-b})\le n^{-b}$ implies that
\begin{equation}
\label{e:n1plus2b}
8Rkn^{-2}\tmix^{\mathrm{IP}(2)}\left(\frac{n^{-b/2}}{8k}\right) \le 8Rkn^{-2}\tmix^{\mathrm{IP}(2)}\left(\frac{1}{8n^{1+b/2}}\right) \le n^{-b/2} . 
 \end{equation}
Hence by \eqref{e:tvbound} we have that
\[\tmix^{\mathrm{IP}(k)}(n^{-b/2}) \le C\tmix^{\mathrm{IP}(2)}(n^{-b/2}).  \]
Hence  \eqref{e:Caputomix} for $\eps \le n^{-b/2}$  follows from \eqref{e:smalleps2} and \eqref{e:Caputo1} follows from \eqref{e:smalleps3}. To obtain \eqref{e:Caputomix}  for  $\eps \in (n^{-b/2},\frac{1}{4} \wedge \frac{1}{k})$, we note that similarly to  \eqref{e:n1plus2b} we have that if $n \ge n_0(b)$ then 
\[8Rkn^{-2}\tmix^{\mathrm{IP}(2)}\left(\frac{\eps}{8k}\right) \le  n^{-b/2} . \]
Hence by \eqref{e:tvbound} we have that $\tmix^{\mathrm{IP}(k)}(\eps) \le C\tmix^{\mathrm{IP}(2)}(\eps) $.   
\end{proof}
\begin{rmk}\label{rmk:naive}
Under stronger conditions on $k$ the argument just presented  could be simplified. Consider, for example, a burn-in period of duration  $s=C\tmix^{\mathrm{IP(2)}}(1/k)$. Then it can be shown that if $k \delta \le \frac{1}{32} $, with probability bounded away from zero, no pair of particles interact  during time interval $[s,2s]$. This would lead to the bound $\tmix^{\mathrm{IP(k)}} \lesssim \tmix^{\mathrm{IP(2)}}(1/k)$. 
\end{rmk}
\section{Proof of bound for vertex-transitive graphs: Theorem~ \ref{thm:transitive}}
\begin{proof}[Proof of Theorem~\ref{thm:transitive}]
If $n^{1/12}<k\le n^a$ we appeal to Theorem 1.2 of \cite{HPexclusion}  (which holds also for interchange provided $k\le n/2$, despite being stated for the exclusion process) to obtain $t_\mathrm{mix}^{\mathrm{IP}(k)}(\varepsilon)\lesssim_a t_\mathrm{rel}\log (n/\varepsilon)$  for all $n$ sufficiently large (depending on $a$ -- to guarantee $n^a\le n/2$). On the other hand we have (using \eqref{e:RW1mixtrel}) $t_\mathrm{mix}^{\mathrm{RW}(1)}(\varepsilon)\gtrsim t_\mathrm{rel}|\log (2\varepsilon)|\gtrsim \trel\log(n/\varepsilon)$, for $\varepsilon\le k^{-1}$ and $k> n^{1/12}$ which completes the proof in this regime. So for the rest of the proof we suppose that $k\le n^{1/12}$.

We consider two cases depending on the growth rate of the diameter $D$ of the vertex-transitive graph $G$. Suppose first that $D>n^{1/3}$ so that for all $n$ sufficiently large (depending on $d$), we have $D\ge (n/d)^{1/4}$.\footnote{If $d \ge n^{1/3}$ one can bound the mixing time e.g.\ using the bound from \cite{Alon}, and for such large $d$ this bound can be completely absorbed into the constant which depend on $d$.}  Then by Corollary 2.8 of \cite{TandT} we know that there exist constants $A,B>0$ such that (provided $n\ge n_0(d)$) $G$ has $(A,B)$-moderate growth (in the sense described in \cite{moderate}). It then follows from Proposition 11.1 of \cite{HPexclusion} (the proposition as stated there is for the $1/4$-mixing time of the exclusion process but the upper-bound holds also for interchange on $k\le n/2$ particles and the proof carries over to the $\varepsilon$-mixing time for $\varepsilon\le k^{-1}$) that, uniformly in $k\le n^{1/12}$, $t_\mathrm{mix}^{\mathrm{IP}(k)}(\varepsilon)\lesssim_d D^2\log(1/\varepsilon)\asymp \trel\log(1/\varepsilon)\lesssim_d t_\mathrm{mix}^{\mathrm{RW}(1)}(\varepsilon)$ provided $n\ge n_0(d)$.

Now suppose that $D\le n^{1/3}$. We use the following result on vertex transitive graphs (to appear in a future work of Nathana\"el Berestycki, the first author, and Lucas Teyssier; as mentioned in the introduction, the credit for this result is  due to Tessera and Tointon \cite{TandTsharp}, as this bound is a  consequence of their bound on the isoperimetric profile and the generic evolving sets bound on the return probability \cite{evolving}):
\begin{prop}[Berestycki, Hermon, Teyssier]
\label{p:BHT}
There exist $C(d,m)$ such that for over all vertex-transitive graphs of size $n$ and degree $d$ satisfying that $n \ge D^q$ with $\lfloor q \rfloor=m$, writing $R:=D^{q-m}$, for all $t \le D^2$ and  every vertex $x$ we have that
\[p_t(x,x)\le C(d,m)  \left( \frac{1}{t^{(m+1)/2}} \vee \frac{1}{Rt^{m/2}} \right). \] 

In particular, if $n\ge D^3$, then uniformly over $t\le D^2$, $ p_t(x,x)\lesssim_d t^{-3/2}$.
\end{prop}

In order to apply Theorem~\ref{thm:transiance} we also need to verify that $t_\mathrm{rel}^{\mathrm{RW}(1)}\le n^{1-2a}$ for some $a>0$. The diameter bound  on the relaxation time gives $t_\mathrm{rel}^{\mathrm{RW}(1)}\le 2D^2\le 2 n^{2/3}\le n^{3/4}$ for $n$ sufficiently large \cite[Theorem 13.26]{levin} (because $r_e\equiv 1$ rather than $r_e \equiv 1/d$ the factor $d$  in the reference can be removed). Thus by Theorem~\ref{thm:transiance} (and using the remark that follows it) we deduce that if $D^2\ge 2\alpha t_\mathrm{mix}^{\mathrm{RW}(1)}(\varepsilon)$ (where $\alpha=\alpha(d)$ is determined in the proof of Theorem~\ref{thm:transiance}) then for all $n$ sufficiently large (depending on $d$), and all $3\le k\le n^{1/12}$ we have $t_\mathrm{mix}^{\mathrm{IP}(k)}(\varepsilon)\lesssim_d t_\mathrm{mix}^{\mathrm{RW}(1)}(\varepsilon)$. 
On the other hand, if $D^2<2\alpha t_\mathrm{mix}^{\mathrm{RW}(1)}(\varepsilon)=:2s$, then we apply Proposition \ref{p:BHT}   at time $D^2$ (namely, $\max_x(p_{D^2}(x,x)\lesssim D^{-3} $) together with the Poincar\'e inequality (between times $D^2$ and $s \wedge D^2$) and the diameter upper bound on the relaxation time to deduce that
\[\max_x(p_{2s}(x,x)-1/n)\lesssim_d  \left(1\wedge e^{-2(s-D^2)/t_\mathrm{rel}}\right)D^{-3}\le \left(1\wedge e^{-(s-D^2)/D^2}\right)D^{-3}.\]
To complete the proof it suffices to show that the right-hand side above is $\lesssim_d (2s)^{-(1+\theta)}$ for some $\theta>0$. If $s<D^2$ then $D^{-3}\lesssim (2s)^{-3/2}$ as needed. So suppose $s>D^2$. 
We can work on the additional assumption that $D^2> D_0$ for any constant $D_0$ since, for fixed $d$, there are finitely many graphs with $D^2\le D_0$ and (as $r_e\equiv 1$) finitely many Markov processes, and so the result is trivial in this case. 
 It remains to show
\[
 (1+\theta)\log s\le \frac{s}{D^2}+3\log D,
\]
for some $\theta>0$ and $D$ sufficiently large. 
However this is clear: if $3\log D\ge sD^{-2}$ then $\log s\le 2\log D+3\log\log D\le \frac52\log D$ (for $D$ sufficiently large), so $\frac65\log s\le 3\log D\le sD^{-2}+3\log D;$ if $3\log D< sD^{-2}$ then (for $D$ sufficiently large)\[\frac{5}{4}\log s<\frac{5}{4}\frac{s\log(3D^2\log D)}{3D^2\log D}<\frac{s}{D^2}.\qedhere
\]
\end{proof}
\section{Proof of Theorem~\ref{thm:transiance}}
\begin{proof}[Proof of Theorem~\ref{thm:transiance}]
We first consider the case $\varepsilon\in[n^{-a},\frac1{k}\wedge\frac14]$. 
In order to apply Theorem~\ref{T:main}, we  show that $\delta\le \varepsilon/(2k)$, where (as defined in the statement of Theorem~\ref{T:main}) $\delta:=8Rkn^{-2}\tmix^{\mathrm{IP}(2)}(\frac{\varepsilon}{8k})$. Under the maximal degree assumption, we have the bound $R:=n^{-1}\sum_e r_e|e|(|e|-1) \le 2d$ and so $\delta\le 16 dkn^{-2}\tmix^{\mathrm{IP}(2)}(\frac{\varepsilon}{8k})$. Using the general relations between $\tmix$ and $\trel$ (i.e.\ for reversible irreducible chains $t_\mathrm{mix}(\varepsilon)\le t_\mathrm{rel}\log(\frac1{\varepsilon\pi_\mathrm{min}})$) and the Caputo, Liggett and Richthammer Theorem \cite{Caputo} we have $\tmix^{\mathrm{IP}(2)}(\frac{\varepsilon}{8k})\le  t_\mathrm{rel}^\mathrm{IP(2)}\log(8kn^2/\varepsilon)= t_\mathrm{rel}^\mathrm{RW(1)}\log(8kn^2/\varepsilon)$. By the assumption on $\trel^\mathrm{RW(1)}$ we deduce that $\delta\le 16dkn^{-2}\log(8kn^2/\varepsilon)n^{1-2a}$. Now $k\le n^a$ and $\varepsilon\ge n^{-a}$, hence $\delta\le 16d k^{-1}n^{-1}\log (8kn^2/\varepsilon)\le \varepsilon/(2k),$ provided $n$ is sufficiently large (depending on $d$ and $a$). Thus it follows by Theorem~\ref{T:main} that there exists a universal $C>0$ such that for each $\varepsilon\in[n^{-a},\frac1{k}\wedge\frac14]$, \begin{align}\label{e:tmixkto2}
t_\mathrm{mix}^{\mathrm{IP}(k)}(\varepsilon)\le C t_\mathrm{mix}^\mathrm{IP(2)}(\varepsilon).\end{align} 

Next, set $s=\alpha t_\mathrm{mix}^\mathrm{RW(1)}(\varepsilon)$ for some $\alpha=\alpha(c,d,\theta)$ to be determined, and let $N_{(t_1,t_2)}(a,b)$ denote the number of interactions during time interval $(t_1,t_2)$ between particles started (at time 0) from vertices $a$ and $b$. Similarly let $\tilde N_{(t_1,t_2)}(a,b)$ denote the time that these particles are adjacent during interval $(t_1,t_2)$. Then as edges ring at rate 1, $\mathbb{E}[N_{(t_1,t_2)}(a,b)]=\mathbb{E}[\tilde N_{(t_1,t_2)}(a,b)]$ (this follows by the same reasoning as a similar statement in the proof of Lemma 5.9 in \cite{HPexclusion}, i.e.\! by noticing that interactions between $a$ and $b$ do not affect the unordered pair of trajectories $\{I_t(a),I_t(b)\}$). Thus for any $a,b\in V$,
\begin{align}
\mathbb{E}[N_{(s,2s)}(a,b)]&=\mathbb{E}\left[\int_{s}^{2s}\sum_x\sum_{y\sim x}\indic{I_{s}(a,b)=(x,y)}\,\mathrm{d}t\right]\notag\\
&\le\int_{s}^{2s}\sum_x\sum_{y\sim x}\mathbb{P}[I_t(a)\in\{x,y\},\,I_t(b)\in\{x,y\}]\mathrm{d}t\notag\\
&\le \int_{s}^{2s}\sum_x\sum_{y\sim x}\mathbb{P}[I_t(a)\in\{x,y\}]\,\mathbb{P}[I_t(b)\in\{x,y\}]\mathrm{d}t\notag\\
&\le \int_{s}^{2s}\sqrt{\left(\sum_x\sum_{y\sim x}\mathbb{P}[I_t(a)\in\{x,y\}]^2\right)\,\left(\sum_x\sum_{y\sim x}\mathbb{P}[I_t(b)\in\{x,y\}]^2\right)}\mathrm{d}t\notag\\
&\le \int_{s}^{2s}\sqrt{\left(\sum_x\sum_{y\sim x}\mathbb{P}[I_t(a)\in\{x,y\}]^2\right)\,\left(\sum_x\sum_{y\sim x}\mathbb{P}[I_t(b)\in\{x,y\}]^2\right)}\mathrm{d}t\notag\\
&\le \int_{s}^{2s}\sqrt{\left(\sum_x\sum_{y\sim x}\left(2p_t(a,x)^2+2p_t(a,y)^2\right)\right)\,\left(\sum_x\sum_{y\sim x}\left(2p_t(b,x)^2+2p_t(b,y)^2\right)\right)}\mathrm{d}t\notag
\end{align}
where the second inequality follows from the negative correlation property of the exclusion process and the third is by Cauchy-Schwarz. 
Now we observe that by reversibility for each $a\in V$,
\[
\sum_x\sum_{y\sim x}\left(p_t(a,x)^2+p_t(a,y)^2\right)=d\sum_xp_t(a,x)^2+d\sum_yp_t(a,y)^2=2dp_{2t}(a,a).
\]
Using this in the previous display we obtain
\begin{align}
\mathbb{E}[N_{(s,2s)}(a,b)]&\le 4d\int_s^{2s}\max_zp_{2t}(z,z)\mathrm{d}t\le 4d\left(\frac{s}{n}+s\big(\max_z p_{2s}(z,z)-\frac{1}{n}\big)\right).\label{e:EN}
\end{align}

By assumption \ref{e:HKsimpler},
\[
2s\le c^{(1+\theta)^{-1}}\Big(\max_zp_{2s}(z,z)-\frac1{n}\Big)^{-(1+\theta)^{-1}}
\]
from which we obtain
\[
2s\Big(\max_zp_{2s}(z,z)-\frac1{n}\Big)\le c^{(1+\theta)^{-1}} \Big(\max_zp_{2s}(z,z)-\frac1{n}\Big)^{\theta(1-\theta)}.
\]
The Poincar\'e inequality gives that for reversible Markov chains, if $t\ge c_1\trel\log (1/\varepsilon)$ then \[\max_x p_t(x,x)-\frac1{n}\le \varepsilon^{2c_1}.\]
It follows that, as $2s=2\alpha t_\mathrm{mix}^\mathrm{RW(1)}(\varepsilon)\ge 2\alpha\trel^\mathrm{RW(1)}\log (1/\varepsilon)$, we have 
\[\max_z p_{2s}(z,z)-\frac1{n}\le \varepsilon^{4\alpha}\]and so \[
2s\Big(\max_zp_{2s}(z,z)-\frac1{n}\Big)\le c^{(1+\theta)^{-1}}\varepsilon^{4\alpha\theta(1-\theta)}.
\]
Using this in \eqref{e:EN} we obtain
\begin{align*}
\mathbb{E}[N_{(s,2s)}(a,b)]&\le \frac{4ds}{n}+2dc^{(1+\theta)^{-1}}\varepsilon^{4\alpha\theta(1-\theta)}= \frac{4d}{n}\alpha t_\mathrm{mix}^\mathrm{RW(1)}(\varepsilon)+2dc^{(1+\theta)^{-1}}\varepsilon^{4\alpha\theta(1-\theta)}\\
&\le \frac{4d}{n}\alpha\trel^\mathrm{RW(1)}\log (n/\varepsilon)+\frac{\varepsilon}{16}\le \frac{4d}{n}\alpha n^{1-2a}\log (n/\varepsilon)+\frac{\varepsilon}{16}\\
&\le \frac{\varepsilon}{8},
\end{align*}
using $\varepsilon\ge n^{-a}$, choosing an appropriately large $\alpha=\alpha(c,d,\theta)$, and provided $n$ is sufficiently large (depending on $a$ and $d$).

To complete the proof for this case, we apply Lemma~\ref{L:submulti} (taking the $k$ there to be 2), which gives that 
\[
\bar d_2(2s)\le 2\max_{{\bf x}\in (V)_2}(1-\mathbb{P}[J_s({\bf x})])+\bar d_1(s)\le 2\max_{{\bf x}\in (V)_2} \mathbb{E}[N_{(s,2s)}({\bf x})]+\frac{\varepsilon}{4}\le \frac{\varepsilon}{2},
\]
where we have used Markov's inequality in the second inequality above and possibly increased $\alpha$ so that by submultiplicativity $\bar d_1(s)\le \varepsilon/2$. It follows that  $t_\mathrm{mix}^\mathrm{IP(2)}(\varepsilon/2)\le 2s$ and combining with \eqref{e:tmixkto2} (and using submultiplicativity again) we deduce that there exists a constant $C'=C'(\alpha)$ such that provided $n$ is sufficiently large (depending on $a$ and $d$),
\[
t_\mathrm{mix}^{\mathrm{IP}(k)}(\varepsilon)\le C' t_\mathrm{mix}^\mathrm{RW(1)}(\varepsilon),
\]for each $\varepsilon\in[n^{-a},\frac1{k}\wedge\frac14]$.
It remains to consider the case $\varepsilon<n^{-a}$. By submultiplicativity, $\tmix^{\mathrm{IP}(k)}(\varepsilon)\lesssim \tmix^{\mathrm{IP}(k)}(n^{-a})\log_{n^a}(1/\varepsilon),$ and so by the result just demonstrated, we have $\tmix^{\mathrm{IP}(k)}(\varepsilon)\lesssim_\alpha \tmix^{\mathrm{RW}(1)}(n^{-a})\log_{n^a}(1/\varepsilon)\lesssim \trel \log n \log_{n^a}(1/\varepsilon)\lesssim \trel\log(1/\varepsilon)\lesssim \tmix^{\mathrm{RW}(1)}(\varepsilon)$ for $n$ sufficiently large.
\end{proof}

\section{Comparison of Dirichlet forms: Proof of Theorem~\ref{thm:Caputo}}
The majority of this section is devoted to the proof of~\eqref{e:trel2trel1}. Once we have this,  \eqref{e:Caputo12uniform} follows easily by combining it with \eqref{e:Caputo1}.  So we can now focus on the proof of~\eqref{e:trel2trel1}.

Since RW(1) and RW(2) have the same spectral-gap, our goal is to compare the Dirichlet forms associated with IP(2)  and RW(2). However the two do not have the same state space. To rectify this, we consider the auxiliary process   Q(2) -- the process obtained from RW(2) by observing it only when particles are at different locations (that is, we remove the times at which they are at the same location).  
We shall compare Dirichlet forms associated with IP(2) and Q(2). This suffices because of the following lemma.
\begin{lemma}
\label{lem:Q2RW2} For all finite hypergraphs $G=(V,E)$ and all rates $(r_e:r \in E)$ such that the associated uniform interchange process   with two particles is irreducible we have that $Q(2)$ is reversible and its stationary distribution is the uniform distribution on $(V)_2$. Moreover
\begin{equation}
\label{e:Q2RW2trel}
\trel^{Q(2)} \le  \trel^{\mathrm{RW}(2)}. 
\end{equation}
\end{lemma}
\begin{proof}
It is easy to check that the generator of $Q(2)$ is symmetric and hence indeed it is reversible w.r.t.\ the uniform distribution on $(V)_2$, which below we denote by $\mu$.

For \eqref{e:Q2RW2trel} cf.\ \cite[Theorem 13.16]{levin} (the proof is written in discrete time, but only minor adaptations are needed for the continuous time case). 
\end{proof}

Write $q^{\mathrm{IP}(2)}(\cdot,\cdot)$ for the transition rates of IP(2). We first note that the Dirichlet form associated with IP(2) can be written as
\begin{align}
\mathcal{E}^{\mathrm{IP}(2)}(f,f)=&\frac12\sum_{{\bf a}\in(V)_2}\sum_{\substack{{\bf b}\in (V)_2\\{\bf b}\neq{\bf a}}}\frac1{\binom{n}{2}}q^{\mathrm{IP}(2)}({\bf a},{\bf b})\big(f({\bf a})-f({\bf b})\big)^2\notag\\=& 
\frac12\sum_{{\bf a}\in(V)_2}\sum_{\substack{{\bf b}\in (V)_2\\{\bf b}\neq{\bf a}}}\frac1{\binom{n}{2}}\big(f({\bf a})-f({\bf b})\big)^2\sum_{e:\,{\bf a},{\bf b}\in(e)_2}\frac{r_e}{|e|(|e|-1)}\label{eq:EIP2_1}\\
&+\frac12\sum_{{\bf a}\in(V)_2}\sum_{\substack{{\bf b}\in (V)_2\\{\bf b}\neq{\bf a}\\{\bf b}(1)={\bf a}(1)}}\frac1{\binom{n}{2}}\big(f({\bf a})-f({\bf b})\big)^2\sum_{\substack{e:\,{\bf a}(1)\notin e\\{\bf a}(2),{\bf b}(2)\in e}}\frac{r_e}{|e|}\label{eq:EIP2_2}\\
&+\frac12\sum_{{\bf a}\in(V)_2}\sum_{\substack{{\bf b}\in (V)_2\\{\bf b}\neq{\bf a}\\{\bf b}(2)={\bf a}(2)}}\frac1{\binom{n}{2}}\big(f({\bf a})-f({\bf b})\big)^2\sum_{\substack{e:\,{\bf a}(2)\notin e\\{\bf a}(1),{\bf b}(1)\in e}}\frac{r_e}{|e|}\label{eq:EIP2_3}.
\end{align}

Write $q^{\mathrm{Q}(2)}(\cdot,\cdot)$ for the transition rates of Q(2). The process Q(2) is reversible with uniform stationary distribution and Dirichlet form
\begin{align}
\mathcal{E}^{\mathrm{Q}(2)}(f,f)=&\frac12\sum_{{\bf x}\in(V)_2}\sum_{\substack{{\bf y}\in(V)_2\\{\bf y}\neq{\bf x}}}\frac1{\binom{n}{2}}q^{\mathrm{Q}(2)}({\bf x},{\bf y})\big(f({\bf x})-f({\bf y})\big)^2\notag\\
\le&\frac12\frac1{\binom{n}{2}}\sum_{{\bf x}\in(V)_2}\sum_{\substack{{\bf y}\in(V)_2\\{\bf y}\neq{\bf x}\\{\bf y}(1)={\bf x}(1)}}q^{\mathrm{Q}(2)}({\bf x},{\bf y})\big(f({\bf x})-f({\bf y})\big)^2\label{eq:EQ2_1}\\
&+\frac12\frac1{\binom{n}{2}}\sum_{{\bf x}\in(V)_2}\sum_{\substack{{\bf y}\in(V)_2\\{\bf y}\neq{\bf x}\\{\bf y}(2)={\bf x}(2)}}q^{\mathrm{Q}(2)}({\bf x},{\bf y})\big(f({\bf x})-f({\bf y})\big)^2\label{eq:EQ2_2}\\
&+\frac12\frac1{\binom{n}{2}}\sum_{{\bf x}\in(V)_2}\sum_{\substack{{\bf y}\in(V)_2\\{\bf y}\neq{\bf x}\\{\bf y}(2)={\bf x}(1)}}q^{\mathrm{Q}(2)}({\bf x},{\bf y})\big(f({\bf x})-f({\bf y})\big)^2\label{eq:EQ2_3}\\
&+\frac12\frac1{\binom{n}{2}}\sum_{{\bf x}\in(V)_2}\sum_{\substack{{\bf y}\in(V)_2\\{\bf y}\neq{\bf x}\\{\bf y}(1)={\bf x}(2)}}q^{\mathrm{Q}(2)}({\bf x},{\bf y})\big(f({\bf x})-f({\bf y})\big)^2.\label{eq:EQ2_4}
\end{align}
(this is an inequality rather than equality because we have, in the last two lines, double-counted the terms where ${\bf y}=({\bf x}(2),{\bf x}(1))$.
We show that each of these terms can be upper-bounded by a linear combination of the terms~\eqref{eq:EIP2_1}-\eqref{eq:EIP2_3}. 
\subsection{Bounding terms~\eqref{eq:EQ2_1} and~\eqref{eq:EQ2_2}}
The two terms~\eqref{eq:EQ2_1} and~\eqref{eq:EQ2_2} can be dealt with similarly (likewise for~\eqref{eq:EQ2_3} and~\eqref{eq:EQ2_4} in the next subsection). Term~\eqref{eq:EQ2_1} is equal to
\begin{align}\notag
&\frac12\frac1{\binom{n}{2}}\sum_{{\bf x}\in(V)_2}\sum_{\substack{{\bf y}\in(V)_2\\{\bf y}\neq{\bf x}\\{\bf y}(1)={\bf x}(1)}}q^{\mathrm{Q}(2)}({\bf x},{\bf y})\big(f({\bf x})-f({\bf y})\big)^2\\
&=\frac12\frac1{\binom{n}{2}}\sum_{{\bf x}\in(V)_2}\sum_{\substack{{\bf y}\in(V)_2\\{\bf y}\neq{\bf x}\\{\bf y}(1)={\bf x}(1)}}\big(f({\bf x})-f({\bf y})\big)^2\sum_{\substack{e:\,{\bf x}(1)\notin e\\{\bf x}(2),{\bf y}(2)\in e}}\frac{r_e}{|e|}\label{eq:simple1}\\
&+\frac12\frac1{\binom{n}{2}}\sum_{{\bf x}\in(V)_2}\sum_{\substack{{\bf y}\in(V)_2\\{\bf y}\neq{\bf x}\\{\bf y}(1)={\bf x}(1)}}\big(f({\bf x})-f({\bf y})\big)^2
\sum_{\substack{e:\,{\bf x}(1)\in e\\{\bf x}(2),{\bf y}(2)\in e}}\frac{r_e}{|e|}\label{eq:simple2}\\
&
+\frac12\frac1{\binom{n}{2}}\sum_{{\bf x}\in(V)_2}\sum_{\substack{{\bf y}\in(V)_2\\{\bf y}\neq{\bf x}\\{\bf y}(1)={\bf x}(1)}}\big(f({\bf x})-f({\bf y})\big)^2
\sum_{e:\,{\bf x}\in(e)_2}\frac{r_e}{2|e|}\sum_{e':\,{\bf x}(1),{\bf y}(2)\in e'}\frac{r_{e'}}{|e'|-1}\left(\sum_{\bar e:\,{\bf x}(1)\in\bar e}r_{\bar e}\right)^{-1}.\label{eq:compound}
\end{align}
Terms~\eqref{eq:simple1} and~\eqref{eq:simple2} correspond to the situation in which particle 2 jumps directly from ${\bf x}(2)$ to ${\bf y}(2)$. Note that \eqref{eq:simple1} is the same as~\eqref{eq:EIP2_2}. Term~\eqref{eq:compound} corresponds to particle 2 jumping first to the vertex occupied by particle 1 (which is not observed by Q(2) as at this point the two particles are at the same location), and then the next edge that rings containing ${\bf x}(1)$ (in RW(2)) is an edge which also has ${\bf y}(2)$ on it, and when it rings particle 2 jumps to ${\bf y}(2)$.

Term~\eqref{eq:simple2} requires a little manipulation. We can write it as
\begin{align*}
&\frac12\frac1{\binom{n}{2}}\sum_{{\bf x}\in(V)_2}\sum_{\substack{{\bf y}\in(V)_2\\{\bf y}\neq{\bf x}\\{\bf y}(1)={\bf x}(1)}}
\sum_{\substack{e:\,{\bf x}(1)\in e\\{\bf x}(2),{\bf y}(2)\in e}}\frac{r_e}{|e|}\frac{1}{\binom{|e|}{2}}\sum_{{\bf z}\in(e)_2}\big(f({\bf x})-f({\bf z})+f({\bf z})-f({\bf y})\big)^2\\
&\le\frac12\frac1{\binom{n}{2}}\sum_{{\bf x}\in(V)_2}\sum_{\substack{{\bf y}\in(V)_2\\{\bf y}\neq{\bf x}\\{\bf y}(1)={\bf x}(1)}}
\sum_{\substack{e:\,{\bf x}(1)\in e\\{\bf x}(2),{\bf y}(2)\in e}}\frac{r_e}{|e|}\frac{1}{\binom{|e|}{2}}\sum_{{\bf z}\in(e)_2}2\big[\big(f({\bf x})-f({\bf z})\big)^2+\big(f({\bf z})-f({\bf y})\big)^2\big]\\
&=\frac1{\binom{n}{2}}\sum_{{\bf x}\in(V)_2}\sum_{\substack{{\bf y}\in(V)_2\\{\bf y}\neq{\bf x}\\{\bf y}(1)={\bf x}(1)}}
\sum_{\substack{e:\,{\bf x}(1)\in e\\{\bf x}(2),{\bf y}(2)\in e}}\frac{r_e}{|e|}\frac{1}{\binom{|e|}{2}}\sum_{{\bf z}\in(e)_2}2\big(f({\bf x})-f({\bf z})\big)^2\\
&=\frac1{\binom{n}{2}}\sum_{{\bf x},{\bf z}\in(V)_2}\sum_{e:{\bf x},{\bf z}\in (e)_2}\frac{r_e}{|e|}\frac1{\binom{|e|}{2}}2\big(f({\bf x})-f({\bf z})\big)^2(|e|-1)\\
&=\frac{4}{\binom{n}{2}}\sum_{{\bf x},{\bf z}\in (V)_2}\big(f({\bf x})-f({\bf z})\big)^2\sum_{e:{\bf x},{\bf z}\in (e)_2}\frac{r_e}{|e|^2}.
\end{align*}
Observe that this is at most 8 times~\eqref{eq:EIP2_1}. 

Term \eqref{eq:compound} requires even more manipulation. For fixed ${\bf x}$ and ${\bf y}$, set ${\bf w}=({\bf y}(2),{\bf x}(2))$, ${\bf z}=({\bf y}(2),{\bf y}(1))$ and $\bar{\bf z}=({\bf z}(2),{\bf z}(1))$. Using that 
\[
\big(f({\bf x})-f({\bf y})\big)^2\le 3\left[\big(f({\bf x})-f({\bf w})\big)^2+\big(f({\bf w})-f({\bf z})\big)^2+\big(f({\bf z})-f({\bf y})\big)^2\right]
\]we decompose~\eqref{eq:compound} into three terms:
\begin{align}
&\frac32\frac1{\binom{n}{2}}\sum_{{\bf x}\in(V)_2}\sum_{\substack{{\bf w}\in(V)_2\\{\bf w}\neq{\bf x}\\{\bf w}(2)={\bf x}(2)}}\big(f({\bf x})-f({\bf w})\big)^2
\sum_{e:\,{\bf x}\in(e)_2}\frac{r_e}{2|e|}\sum_{e':\,{\bf x}(1),{\bf w}(1)\in e'}\frac{r_{e'}}{|e'|-1}\left(\sum_{\bar e:\,{\bf x}(1)\in\bar e}r_{\bar e}\right)^{-1}\label{eq:decomp1}\\
&+\frac32\frac1{\binom{n}{2}}\sum_{{\bf w}\in(V)_2}\sum_{\substack{{\bf z}\in(V)_2\\{\bf w}\neq{\bf z}\\{\bf w}(1)={\bf z}(1)}}\big(f({\bf w})-f({\bf z})\big)^2
\sum_{e:\,{\bf w}(2),{\bf z}(2)\in e}\frac{r_e}{2|e|}\sum_{e':\,{\bf z}\in (e')_2}\frac{r_{e'}}{|e'|-1}\left(\sum_{\bar e:\,{\bf z}(2)\in\bar e}r_{\bar e}\right)^{-1}\label{eq:decomp2}\\
&+\frac32\frac1{\binom{n}{2}}\sum_{{\bf z}\in(V)_2}\big(f({\bf z})-f(\bar{\bf z})\big)^2
\sum_{{\bf x}(2)\in V}\sum_{e:\,{\bf z}(2),{\bf x}(2)\in e}\frac{r_e}{2|e|}\sum_{e':\,{\bf z}\in (e')_2}\frac{r_{e'}}{|e'|-1}\left(\sum_{\bar e:\,{\bf z}(2)\in\bar e}r_{\bar e}\right)^{-1}.\label{eq:decomp3}
\end{align}
Terms~\eqref{eq:decomp1} and~\eqref{eq:decomp2} can be treated similarly. Firstly, \eqref{eq:decomp1} can be split into two terms depending on whether ${\bf x}(2)$ is in $e'$:
\begin{align}
&\frac32\frac1{\binom{n}{2}}\sum_{{\bf x}\in(V)_2}\sum_{\substack{{\bf w}\in(V)_2\\{\bf w}\neq{\bf x}\\{\bf w}(2)={\bf x}(2)}}\big(f({\bf x})-f({\bf w})\big)^2
\sum_{e:\,{\bf x}\in(e)_2}\frac{r_e}{2|e|}\sum_{\substack{e':\,{\bf x}(1),{\bf w}(1)\in e'\\{\bf x}(2)\in e'}}\frac{r_{e'}}{|e'|-1}\left(\sum_{\bar e:\,{\bf x}(1)\in\bar e}r_{\bar e}\right)^{-1}\notag\\
&+\frac32\frac1{\binom{n}{2}}\sum_{{\bf x}\in(V)_2}\sum_{\substack{{\bf w}\in(V)_2\\{\bf w}\neq{\bf x}\\{\bf w}(2)={\bf x}(2)}}\big(f({\bf x})-f({\bf w})\big)^2
\sum_{e:\,{\bf x}\in(e)_2}\frac{r_e}{2|e|}\sum_{\substack{e':\,{\bf x}(1),{\bf w}(1)\in e'\\{\bf x}(2)\notin e'}}\frac{r_{e'}}{|e'|-1}\left(\sum_{\bar e:\,{\bf x}(1)\in\bar e}r_{\bar e}\right)^{-1}\notag\\
&\le\frac34\frac1{\binom{n}{2}}\sum_{{\bf x}\in(V)_2}\sum_{\substack{{\bf w}\in(V)_2\\{\bf w}\neq{\bf x}\\{\bf w}(2)={\bf x}(2)}}\big(f({\bf x})-f({\bf w})\big)^2
\sum_{\substack{e':\,{\bf x}(1),{\bf w}(1)\in e'\\{\bf x}(2)\in e'}}\frac{r_{e'}}{|e'|-1}\label{eq:decomp1split1}\\
&\phantom{le}+\frac34\frac1{\binom{n}{2}}\sum_{{\bf x}\in(V)_2}\sum_{\substack{{\bf w}\in(V)_2\\{\bf w}\neq{\bf x}\\{\bf w}(2)={\bf x}(2)}}\big(f({\bf x})-f({\bf w})\big)^2
\sum_{\substack{e':\,{\bf x}(1),{\bf w}(1)\in e'\\{\bf x}(2)\notin e'}}\frac{r_{e'}}{|e'|-1}\label{eq:decomp1split2}
\end{align}
Now, \eqref{eq:decomp1split1} can be bounded in exactly the same way as~\eqref{eq:simple2}, and so (since we have a factor of $\frac34$ instead of $\frac12$ as in~\eqref{eq:simple2}) it is at most 12 times~\eqref{eq:EIP2_1}. Term~\eqref{eq:decomp1split2} is at most $3$ times~\eqref{eq:EIP2_3}, and thus~\eqref{eq:decomp1} is at most $12\mathcal{E}^{\mathrm{IP}(2)}(f,f)$. A similar argument gives the same bound for~\eqref{eq:decomp2} (splitting it depending on whether ${\bf w}(1)$ is in $e$). 

For~\eqref{eq:decomp3}, we can write it as
\begin{align}
&\frac32\frac1{\binom{n}{2}}\sum_{{\bf z}\in(V)_2}\big(f({\bf z})-f(\bar{\bf z})\big)^2
\sum_{e:\,{\bf z}(2)\in e}\frac{r_e}{2}\sum_{e':\,{\bf z}\in (e')_2}\frac{r_{e'}}{|e'|-1}\left(\sum_{\bar e:\,{\bf z}(2)\in\bar e}r_{\bar e}\right)^{-1}\notag\\
&=\frac34\frac1{\binom{n}{2}}\sum_{{\bf z}\in(V)_2}\big(f({\bf z})-f(\bar{\bf z})\big)^2
\sum_{e:\,{\bf z}\in (e)_2}\frac{r_{e}}{|e|-1}\notag\\
&\le \frac32\frac1{\binom{n}{2}}\sum_{{\bf z}\in(V)_2}\sum_{e:\,{\bf z}\in (e)_2}\frac{r_{e}}{|e|-1}\frac1{\binom{|e|}{2}}\sum_{{\bf v}\in (e)_2}
\Big\{\big(f({\bf z})-f({\bf v})\big)^2+\big(f({\bf v})-f(\bar{\bf z})\big)^2\Big\}
\notag\\
&= \frac{3}{\binom{n}{2}}\sum_{{\bf z}\in(V)_2}\sum_{\substack{{\bf v}\in (V)_2\\{\bf v}\neq{\bf z}}}\big(f({\bf z})-f({\bf v})\big)^2\sum_{e:\,{\bf z},{\bf v}\in (e)_2}\frac{r_{e}}{|e|-1}\frac1{\binom{|e|}{2}}\notag,
\end{align}
which is clearly bounded by $12$ times~\eqref{eq:EIP2_1}.

Adding, we see that we can bound~\eqref{eq:compound} by $36\mathcal{E}^{\mathrm{IP}(2)}(f,f)$ and thus overall term~\eqref{eq:EQ2_1} is at most $44\mathcal{E}^{\mathrm{IP}(2)}(f,f)$. The same is clearly true for~\eqref{eq:EQ2_2}. 
\subsection{Bounding terms \eqref{eq:EQ2_3} and~\eqref{eq:EQ2_4}}
It remains to consider terms~\eqref{eq:EQ2_3} and~\eqref{eq:EQ2_4}. Term~\eqref{eq:EQ2_3} is 
\begin{align}
&\frac12\frac1{\binom{n}{2}}\sum_{{\bf x}\in(V)_2}\sum_{\substack{{\bf y}\in(V)_2\\{\bf y}\neq{\bf x}\\{\bf y}(2)={\bf x}(1)}}q^{\mathrm{Q}(2)}({\bf x},{\bf y})\big(f({\bf x})-f({\bf y})\big)^2\notag\\
&=\frac12\frac1{\binom{n}{2}}\sum_{{\bf x}\in(V)_2}\sum_{\substack{{\bf y}\in(V)_2\\{\bf y}\neq{\bf x}\\{\bf y}(2)={\bf x}(1)}}\big(f({\bf x})-f({\bf y})\big)^2
\sum_{e:\,{\bf x}\in(e)_2}\frac{r_e}{|e|}\sum_{e':{\bf y}\in(e')_2}\frac{r_{e'}}{2(|e'|-1)}\left(\sum_{\bar e:\,{\bf x}(1)\in\bar e}r_{\bar e}\right)^{-1}\label{eq:term3}.
\end{align}
 We again must manipulate term \eqref{eq:term3} a little. For each fixed ${\bf x}$, ${\bf y}$ appearing in the sums, we define ${\bf z}=({\bf y}(1),{\bf x}(2))$, and then we can upper-bound \eqref{eq:term3} by
\begin{align}
&\frac12\frac1{\binom{n}{2}}\sum_{{\bf x}\in(V)_2}\sum_{\substack{{\bf y}\in(V)_2\\{\bf y}\neq{\bf x}\\{\bf y}(2)={\bf x}(1)}}2\left[\big(f({\bf x})-f({\bf z})\big)^2+\big(f({\bf z})-f({\bf y})\big)^2\right]
\sum_{\substack{e:\,{\bf x}\in(e)_2}}\frac{r_e}{|e|}\sum_{e':{\bf y}\in(e')_2}\frac{r_{e'}}{2(|e'|-1)}\left(\sum_{\bar e:\,{\bf x}(1)\in\bar e}r_{\bar e}\right)^{-1}\notag\\
&=\frac12\frac1{\binom{n}{2}}\sum_{{\bf x}\in(V)_2}\sum_{\substack{{\bf z}\in(V)_2\\{\bf z}\neq{\bf x}\\{\bf z}(2)={\bf x}(2)}}2\big(f({\bf x})-f({\bf z})\big)^2
\sum_{\substack{e:\,{\bf x}\in(e)_2}}\frac{r_e}{|e|}\sum_{e':{\bf z}(1),{\bf x}(1)\in e'}\frac{r_{e'}}{2(|e'|-1)}\left(\sum_{\bar e:\,{\bf x}(1)\in\bar e}r_{\bar e}\right)^{-1}\notag\\
&\phantom{=}+\frac12\frac1{\binom{n}{2}}\sum_{{\bf z}\in(V)_2}\sum_{\substack{{\bf y}\in(V)_2\\{\bf y}\neq{\bf z}\\{\bf y}(1)={\bf z}(1)}}2\big(f({\bf z})-f({\bf y})\big)^2
\sum_{\substack{e:\,{\bf y}(2),{\bf z}(2)\in e}}\frac{r_e}{|e|}\sum_{e':{\bf y}\in(e')_2}\frac{r_{e'}}{2(|e'|-1)}\left(\sum_{\bar e:\,{\bf y}(2)\in\bar e}r_{\bar e}\right)^{-1}\notag
\end{align}
\begin{align}
&=\frac12\frac1{\binom{n}{2}}\sum_{{\bf x}\in(V)_2}\sum_{\substack{{\bf z}\in(V)_2\\{\bf z}\neq{\bf x}\\{\bf z}(2)={\bf x}(2)}}2\big(f({\bf x})-f({\bf z})\big)^2
\sum_{\substack{e:\,{\bf x}\in(e)_2}}\frac{r_e}{|e|}\sum_{\substack{e':{\bf z}(1),{\bf x}(1)\in e'\\{\bf z}(2)\in e'}}\frac{r_{e'}}{2(|e'|-1)}\left(\sum_{\bar e:\,{\bf x}(1)\in\bar e}r_{\bar e}\right)^{-1}\label{eq:term3pt1}\\
&\phantom{=}+\frac12\frac1{\binom{n}{2}}\sum_{{\bf x}\in(V)_2}\sum_{\substack{{\bf z}\in(V)_2\\{\bf z}\neq{\bf x}\\{\bf z}(2)={\bf x}(2)}}2\big(f({\bf x})-f({\bf z})\big)^2
\sum_{\substack{e:\,{\bf x}\in(e)_2}}\frac{r_e}{|e|}\sum_{\substack{e':{\bf z}(1),{\bf x}(1)\in e'\\{\bf z}(2)\notin e'}}\frac{r_{e'}}{2(|e'|-1)}\left(\sum_{\bar e:\,{\bf x}(1)\in\bar e}r_{\bar e}\right)^{-1}\label{eq:term3pt2}\\
&\phantom{=}+\frac12\frac1{\binom{n}{2}}\sum_{{\bf z}\in(V)_2}\sum_{\substack{{\bf y}\in(V)_2\\{\bf y}\neq{\bf z}\\{\bf y}(1)={\bf z}(1)}}2\big(f({\bf z})-f({\bf y})\big)^2
\sum_{\substack{e:\,{\bf y}(2),{\bf z}(2)\in e\\{\bf y}(1)\in e}}\frac{r_e}{|e|}\sum_{e':{\bf y}\in(e')_2}\frac{r_{e'}}{2(|e'|-1)}\left(\sum_{\bar e:\,{\bf y}(2)\in\bar e}r_{\bar e}\right)^{-1}\label{eq:term3pt3}\\
&\phantom{=}+\frac12\frac1{\binom{n}{2}}\sum_{{\bf z}\in(V)_2}\sum_{\substack{{\bf y}\in(V)_2\\{\bf y}\neq{\bf z}\\{\bf y}(1)={\bf z}(1)}}2\big(f({\bf z})-f({\bf y})\big)^2
\sum_{\substack{e:\,{\bf y}(2),{\bf z}(2)\in e\\{\bf y}(1)\notin e}}\frac{r_e}{|e|}\sum_{e':{\bf y}\in(e')_2}\frac{r_{e'}}{2(|e'|-1)}\left(\sum_{\bar e:\,{\bf y}(2)\in\bar e}r_{\bar e}\right)^{-1}.\label{eq:term3pt4}
\end{align}
Term~\eqref{eq:term3pt3} is at most~\eqref{eq:simple2} which we have already established is at most 8 times~\eqref{eq:EIP2_1}. By similar arguments we can obtain the same bound on~\eqref{eq:term3pt1}. Term~\eqref{eq:term3pt2} is at most~\eqref{eq:EIP2_3} and term~\eqref{eq:term3pt4} is at most~\eqref{eq:EIP2_2}. We deduce that~\eqref{eq:term3}, and hence~\eqref{eq:EQ2_3} is at most $16\mathcal{E}^{\mathrm{IP}(2)}(f,f)$. The same bound can be obtained for~\eqref{eq:EQ2_4}.

Putting the bounds on \eqref{eq:EQ2_1}-\eqref{eq:EQ2_4} together, we obtain that $\mathcal{E}^{\mathrm{Q}(2)}(f,f)\le 120\mathcal{E}^{\mathrm{IP}(2)}(f,f)$.

\nocite{}
\bibliographystyle{plain}
\bibliography{Exclusion}

\end{document}